\documentclass[12pt]{article}
\usepackage{mathrsfs}
\usepackage{amssymb}
\usepackage{array}
\usepackage{color}
\usepackage{amsthm}
\usepackage{amsmath}
\usepackage{graphicx}
\usepackage{CJKutf8}
\usepackage[square, comma, sort&compress, numbers]{natbib}
\numberwithin{equation}{section}

\newtheorem{Theorem}{Theorem}[section]
\newtheorem{Definition}{Definition}[section]
\newtheorem{Lemma}{Lemma}[section]
\newtheorem{Corollary}{Corollary}[section]
\newtheorem{Remark}{Remark}[section]

\begin{document}
	\begin{CJK}{UTF8}{gbsn}
	\allowdisplaybreaks[4]
	\title{The global existence and blowup of the classical solution to the relativistic dust in a FLRW geometry}
	\author{Xianshu Ju$^{1}$, Xiangkai Ke$^{1}$, Changhua Wei$^{1,2,}$\footnote{\footnotesize Corresponding author: chwei@zstu.edu.cn}\\[12pt]
		\emph {\scriptsize $^{1}$Department of Mathematics, Zhejiang Sci-Tech University, Hangzhou, 310018, China.}\\
		\emph {\scriptsize $^{2}$Nanbei Lake Institute for Artificial Intelligence in Medicine, Jiaxing, 314300, China.}}

	\date{}
	
	\maketitle
	\begin{abstract}
		This paper is concerned with the global existence and blowup of the classical solution to the Cauchy problem of  the relativistic Euler equation
		with $ p=0 $ in a fixed Friedmann-Lema\^{\i}tre-Robertson-Walker (FLRW)  spacetime. The aim of this work is to study clearly the effect of the expansion rate of the spacetime on the life span of the classical solution to the pressureless fluid. Since the density and the velocity of the relativistic dust admits the same principal part, we can obtain much more accurate results by the characteristic method rather than energy estimates.
		
	\end{abstract}
	\section{Introduction}
	In the standard coordinates $(t,x^1,\cdots,x^n)$, the relativistic Euler equation evolving in a fixed Lorenzian spacetime $(M, g)$ can be described as
	\begin{equation}\label{1.1}
		\nabla_{\mu}T_{\nu}^{\mu} = 0,
	\end{equation}
	where Greek indices $\mu$ are assigned values from $\{0,1,\cdots,n\}$. The operator $\nabla_{\mu}$ denotes the covariant derivative with respect to the given metric $g=(g_{\mu\nu})$. $T^{\mu}_\nu$ denotes the energy-momentum tensor, which are given by
	\begin{equation*}\label{1.2}
		T^{\mu}_\nu = (\rho c^{2} + p)u_{\nu}u^{\mu} + pg_{\nu}^{\mu},
	\end{equation*}
	where $\rho$ represents the energy density and $p=p(\rho)$ denotes the fluid pressure. $u=(u^{0},u^{1},\cdots,u^{n})$ is the $1+n$-velocity, which is future directed ($u^0>0$) and timelike vector field normalized by
	\begin{equation*}\label{1.3}
		u_{\mu} u^{\mu}=g_{\mu\nu}u^\mu u^\nu=-1.
	\end{equation*}
	
	In this article, we will focus on analyzing the relativistic Euler equation on general expanding Friedmann-Lema\^{\i}tre-Robertson-Walker (FLRW) spacetime $(M,g)$ with
	\begin{equation*}
		M=[0,+\infty)\times \mathbb R^n,
	\end{equation*}
	and
	\begin{equation}\label{1.4}
		g=-dt^2+a^2(t)\sum_{i=1}^n(dx^i)^2,
	\end{equation}
	where $a(t)>0$ is known as the scale factor and has been used to explain the expansion of our universe by cosmologists.
	
	Throughout the whole paper, repeated upper and lower indices are always summed over their ranges and the index is raised and lowered by \eqref{1.4}.
	
	By the hyperbolic nature of the relativistic Euler equation, the classical solution to \eqref{1.1} in Minkowski spacetime $(a(t)=const.)$ with general initial data always blows up in finite time no matter how small and smooth the initial data are. Indeed, so many results have been obtained to study the blowup phenomena of the initial value problem of system \eqref{1.1}. When $n=1$, the research on shock formation has been investigated by many mathematicians, one can refer to \cite{Courant} for more details. When $n>1$, a breakthrough has been made by Christodoulou \cite{Christodoulou}, in which he considered the compressible and irrotational relativistic Euler equations with constant entropy. By combining the techniques of geometry and analysis, Christodoulou understood the mechanism of the formation of singularities and constructed the weak solution (shocks) with Lisibach \cite{Christodoulou-Lisibach} under the assumption of spherical symmetry. Recently, his method has been investigated deeply and applied to various models, see his work with Miao \cite{Christodoulou-Miao} for classical Euler equations, Miao and Yu \cite{Miao-Yu} for a single variational wave equation, Holzegel, Klainerman, Speck and Wong \cite{H-K-S-W} and refernces theirin for more general quasilinear wave equations.

	To serve our purpose of this paper, we first turn to the fully nonlinear future stability of Friedmann-Lema\^{\i}tre-Robertson-Walker (FLRW) solutions with the positive cosmological constant which is an exact solution to the Einstein matter field equation and has been well studied over the last two decades. Ringstr\"{o}m \cite{Ringstrom} investigated the future global non-linear stability in
	the case of Einstein's equations coupled to a non-linear scalar field $T^{\mu\nu}=\partial^{\mu}\Psi\partial^{\nu}\Psi-[\frac{1}{2}g_{\mu\nu}\partial^{\mu}\Psi\partial^{\nu}\Psi+V(\Psi)]g^{\mu\nu}$, under the assumption that $V(0)>0,\, V^{'}(0)=0,\,V^{''}(0)>0$. $V(\Psi)$ plays the role of the positive cosmological constant, which appears in the most of the following works. The main observation of this paper is that the problem under consideration describes the accelerated expansion of the universe. The associated spacetime expansion creates damping effects, resulting in exponential solution decay. Inspired by Ringstr\"{o}m's work, Rodnianski and Speck \cite{Rodnianski} established the future non-linear stability of these FLRW solutions with positive cosmological constant and linear equation of state $p=K\rho$ under the condition $0 < K < 1/3$ and the assumption of zero fluid vorticity. After that,  Had\v{z}i\'c and Speck \cite{H-J} and Speck \cite{Speck1} answered that this future non-linear stability result remains true
	for fluids with non-zero vorticity and also for the equation of state $p = 0$.
	By employing the conformal method, Oliynyk \cite{Oliynyk1} gave an alternative and fast proof for non-linear future stability problems of FLRW solutions based on conformal singular hyperbolic formulations of the Einstein-Euler equations. For the dust case, the stability of the FLRW-spacetimes with more general spatial manifolds were obtained by Friedrich \cite{Friedrich}.
	
	We emphasize that the above results are obtained by assuming a linear equation of state $p=K\rho$ for the perfect fluids. LeFloch and Wei \cite{LeFloch-Wei} investigated the Einstein-Chaplygin fluids with equation for state $p=-\frac{A}{\rho^{\alpha}}$ with $0<\alpha\leq 1$. This model is widely investigated by physicists as a candidate for  dark energy. In \cite{LeFloch-Wei}, they considered the irrotational fluids, such that the matter is indeed a scalar field, under which the mechanism for the accelerated expansion of the spacetime is the negativity of the pressure (an analogue of the positive cosmological constant). Later, Liu and Wei \cite{Liu-Wei} considered the stabilizing effect of exponential inflation on general fluids and provided a criterion for the fluids. And as a byproduct, the results in \cite{Liu-Wei} can be applied to one of the most important fluids such as polytropic gases $p=K\rho^{\gamma}$ in mathematics and physics.

	When the spacetime admits a power law inflation $a(t)=t^{\alpha}$, Fajman, Ofner, Olinyk and Wyatt \cite{F-T, F-T1} considered the stabilization of the fluids $p=K\rho, (0<K<\frac{1}{3})$ when $\alpha=1$ with and without vorticity, respectively. In the recent work of Fajman, Maliborski, Ofner, Oliynyk and Wyatt \cite{FOOW24}, a phase transition from stable to unstable fluid behavior for cosmological spacetimes undergoing decelerated expansion is given. Through analytical and numerical studies, they also get the transaction between two regions is located on the critial line $K_{crit}(\alpha)=1-\frac{2}{3\alpha}$ when $0<\alpha<1$ and $0\leq K\leq \frac{1}{3}$. Later, Fajman, Ofner, Oliynyk and Wyatt \cite{FOOW25} proved the asymptotically stable of the quiet fluid solution for fluids with non-vanishing speed of sound when $K<K_{crit}$. For the fluids with nonlinear equations of state such as Chaplygin gas, Wei in \cite{Wei} gave the relation between the expanding rate and the speed of the fluids.
	
	Before stating the main problems and the main results of this work, we give a classification of the scale factor $a(t)$ according to the expanding rate.
	\begin{Definition} We call the spacetime is expanding if the scale factor $a(t)$ in \eqref{1.4} satisfies
		$ \dot{a}(t) = \frac{da(t)}{dt} > 0$;
		we call it accelerated expansion, if $a(t)$ satisfies further
		$ \ddot{a}(t) = \frac{d^{2}a(t)}{dt^{2}} > 0$.
	\end{Definition}
	
	To study the relationship between the expanding rate of spacetime and the long-time behavior of the relativistic fluids, we mainly consider the following four cases
	\begin{itemize}
		\item  Assumption $H_{1}$: $\dot{a}(t) > 0 $ and $\ddot{a}(t) \geq 0. $
		
		\item Assumption $H_{2}$: $\dot{a}(t) > 0 $ and $\ddot{a}(t) < 0 $, we further assume that there exists a positive constant $ 0 < \delta_0 < 1 $ such that $ a(t)\geq (1+t)^{\frac{1+\delta_0}{2}}$.
		\item  Assumption $H_{3}$: $\dot{a}(t) > 0 $ and $\ddot{a}(t) < 0 $, and there exists a positive constant $ 0<l \leq\frac{1}{2}$ such that $ a(t) \sim (1 + t)^{l}.$
		\item Assumption $ H_{4}$: $a(t) = (1+t)^0\equiv 1 .$
	\end{itemize}
	\begin{Remark}
		The Assumption $H_1$
		indicates that the spacetime is experiencing accelerated expansion. The difference between Assumption $H_2$ and $H_3$ is that $\int_0^{\infty}a^{-2}(t)dt<+\infty$ if $a(t)$ satisfies Assumption $H_2$. When $a(t)\equiv1$, the background spacetime is turned into the Minkowski one.
	\end{Remark}
	
	In this paper, we are interested in the Cauchy problem for the relativistic fluids with  $ p=0 $,  expanding \eqref{1.1} directly, we can get
	\begin{equation}\label{ME1}
		\begin{cases}
			\frac{\partial}{\partial t}\left(\frac{a(t)\rho}{c^{2}-|v|^{2}}\right) + \sum_{i=1}^{n}\frac{\partial}{\partial x^{i}}\left(\frac{\rho v^{i}}{c^{2}-|v|^{2}}\right) + \left(\frac{n\dot{a}(t)\rho}{c^{2}-|v|^{2}} - \frac{\dot{a}(t)\rho}{c^{2}}\right) = 0,\\
			\frac{\partial}{\partial t}\left(\frac{a(t)\rho v^{i}}{c^{2}-|v|^{2}}\right) + \sum_{j=1}^{n}\frac{\partial}{\partial x^{j}}\left(\frac{\rho v^{i}v^{j}}{c^{2}-|v|^{2}}\right) + \left(\frac{n\dot{a}(t)\rho v^{i}}{c^{2}-|v|^{2}}\right) = 0,\\
			t=0 : \rho(0,x^{1}, x^{2},\cdots, x^{n}) = \varepsilon\rho_{0}(x^{1}, x^{2},\cdots, x^{n}), \\
			t=0 : v(0,x^{1},x^{2},\cdots, x^{n}) = \varepsilon v_0(x^{1},x^{2},\cdots, x^{n}),
		\end{cases}
	\end{equation}
	where $ v^{i} = ca(t)\frac{u^{i}}{u^{0}}$ denotes the classical velocity of the fluids and $c$ is a constant representing the speed of light, $ (\rho_0,v_0)\in C^1(\mathbb{R}^n)\times C^2(\mathbb{R}^n) $, and $\rho_0$ has a bounded 
	$C^1$  norm, while $v_0$ has a bounded $C^2$
	  norm, $ \varepsilon $ denotes a small parameter. For the detailed derivation of \eqref{ME1}, one can refer to \cite{LeFloch} or \cite{Huo-Wei, Huo-Wei1}.

	We assume that system \eqref{ME1} admits a classical solution, then we can derive the following relativistic Burgers equations from \eqref{ME1} studied in \cite{LeFloch} and \cite{Huo-Wei, Huo-Wei1}.
	\begin{equation}\label{equation}
		\begin{cases}
			\frac{\partial v^i}{\partial t}+\sum_{j=1}^{n}\frac{v^j}{a(t)}\frac{\partial v^i}{\partial x^j}+\frac{\dot{a}(t)}{a(t)}v^i-\frac{\dot{a}(t)}{a(t)}\frac{|v|^2}{c^2}v^i=0,\quad i=1,2,\cdots,n,\\
			t=0 : v(0,x^{1},x^{2},\cdots, x^{n}) = \varepsilon v_0(x^{1},x^{2},\cdots. x^{n}).
		\end{cases}
	\end{equation}
	
	We do not care about the singular equation caused by $a^{-1}(t)$, so without loss of generality, we assume that $ a(0) = 1 $ in the following discussions.

	%We prescribe the following initial data to \eqref{ME1}:
	%\begin{equation}\label{ID1}
	%	t=0 : v(0,x^{1},x^{2},\cdots, x^{n}) = \epsilon v_0(x^{1},x^{2},\cdots, x^{n}),
	%\end{equation}
	%and
	%\begin{equation}\label{ID2}
	%	t=0 : \rho(0,x^{1}, x^{2},\cdots, x^{n}) = \epsilon\rho_{0}(x^{1}, x^{2},\cdots, x^{n}),
	%\end{equation}
	Define the matrix
	\begin{eqnarray}
		V(t,\alpha)&=&I+\left(\frac{\partial v_0}{\partial \alpha}\right)\frac{c^3\varepsilon}{[c^2-\varepsilon^2|v_0|^2]^{3/2}}\int_{0}^{t}\frac{a^{-2}(s)}{[1+f_0^2(\alpha)a^{-2}(s)]^{3/2}}\mathrm{d}s\nonumber\\
		&-&(|v_0|^2I-v_0^Tv_0)\left(\frac{\partial v_0}{\partial \alpha}\right)\frac{c\varepsilon^3}{[c^2-\varepsilon^2|v_0|^2]^{3/2}}\int_{0}^{t}\frac{a^{-2}(s)-a^{-4}(s)}{[1+f_0^2(\alpha)a^{-2}(s)]^{3/2}}\mathrm{d}s，
	\end{eqnarray}
	where $ f_0(\alpha)=\frac{\varepsilon |v_0(\alpha)|}{\sqrt{c^2-\varepsilon^2|v_0(\alpha)|^2}}$, $I$ denotes the identity matrix.
	
	Our first result is on the relativistic Burgers equation \eqref{equation}.
	\begin{Theorem} \label{thm:1}
		When $ n \geq 1 $, there exists a positive constant $ \varepsilon_{0} $, such that for any $ \varepsilon\in [0, \varepsilon_{0}]$, the
		life span of the Cauchy problem \eqref{equation} satisfies
		\begin{equation}\label{MR1}
			T(\varepsilon)=\left\{
			\begin{aligned}
				&\infty,    &&    \text{if}\; \text{Assumption} \; H_{1}\; \text{or} \; H_{2} \; holds,\\ &\varepsilon^{\frac{P}{2l-1}}-1,&& \text{if} \;\text{Assumption}\; H_{3}  \; \text{holds} \; \text{but} \; l\neq\frac{1}{2} \; or \;  \text{Assumption}\;H_{4} \; \text{holds},\\
				&e^\frac{P}{\varepsilon}-1,&& if \; a(t) = (1+t)^{1/2},\;
			\end{aligned}
			\right.
		\end{equation}
		where $P$ is  a positive constant independent of  $ \varepsilon $.
		
		Furthermore, under the Assumption $  H_{3}$,
		we have
		
		$(1)$ When $ n=1 $, the system admits a globally classical solution if $ \inf_{\alpha}v_0'(\alpha)\geqslant0  $, and if there exists $\alpha$ such that $ \inf_{\alpha}v_0'(\alpha)<0 $, $v_x$ tends to infinity when $t$ tends to $T(\varepsilon)$ defined in \eqref{MR1}.
		
		$ (2) $ When $ n>1 $, the system admits a globally classical solution if and only if $\det\left(V(t,\alpha)\right)>0$ for all $t>0$ and $\alpha\in \mathbb R^n$.
	\end{Theorem}
	\begin{Remark}\label{rem:1.2}
		Here we have some remarks concerning Theorem \ref{thm:1}
		\begin{itemize}
			\item When $ n=1 $, $ V(t,\alpha)=1+\frac{c^3\varepsilon v_0'(\alpha)}{[c^2-\varepsilon^2|v_0|^2]^{3/2}}\int_{0}^{t}\frac{a^{-2}(s)}{[1+f_0^2(\alpha)a^{-2}(s)]^{3/2}}\mathrm{d}s$. We can show clearly that $ \inf_{\alpha}v_0'(\alpha)\geqslant0$ is equivalent to $det(V(t,\alpha))>0$ for all $t>0$ in the proof of Theorem \ref{thm:1} of section \ref{sec:2} under the Assumption $H_3$.
			\item Compared to the results obtained in \cite{Huo-Wei}, we can get the global solution under the Assumption $H_2$ by the method of characteristics, while Huo and Wei \cite{Huo-Wei} only obtain the lower bound of the life span by energy methods.
			\item Under the Assumption $H_4$,  Theorem 3.1 of \cite{Li} give a sufficient and necessary condition for the existence of the unique global solution of the Cauchy problems \eqref{equation}, which states that none of the eigenvalues of the $n\times n$ matrix $\left(\frac{\partial v_0}{\partial \alpha}\right)$ are negative and this condition can ensure the positiveness of $\det\left(V(t,\alpha)\right)>0$ in Theorem \ref{thm:1} when $a(t)=1$.
			\item For the compressible Euler equations evolving in flat spacetime, Serre \cite{Serre} and Grassin-Serre \cite{Grassin-Serre} have proven the global existence of the classical solution when the spectrum of the Jacobian matrix of the initial velocity is non-negative, this assumption forces the particles to spread out.
		\end{itemize}
	\end{Remark}
	It is difficult to give a criterion similar to Theorem 3.1 of \cite{Li} on the initial data $v_0(\alpha)$ to guarantee $\det\left(V(t,\alpha)\right)>0$ for all $t>0$ under the Assumption $H_3$. Here we can give a sufficient condition, which can ensure the blowup of $|\nabla v|$.
	\begin{Theorem}\label{thm:s}
		Under the Assumption $H_3$, if $\displaystyle\inf_{\alpha}\det\left(\frac{\partial v_0}{\partial \alpha}\right)<0$, then $|\nabla v|$ must tend to infinity in finite time. \end{Theorem}
	\begin{Remark}\label{rem:1.4}
		In Theorem \ref{thm:s},  $\displaystyle\inf_{\alpha}\det\left(\frac{\partial v_0}{\partial \alpha}\right)<0$ is only a sufficient condition. When $\displaystyle\inf_{\alpha}\det\left(\frac{\partial v_0}{\partial \alpha}\right)\geq0$, we can also find different initial data such that the classical solution exists globally or blows up in finite time.
	\end{Remark}
	
	We next consider the whole system \eqref{ME1}, in the existence domain of the classical solution, the system \eqref{ME1} can be equivalently rewritten as
	\begin{equation}\label{ME2}
		\begin{cases}
			\frac{\partial\rho}{\partial t}+\sum_{i=1}^{n}\left(\frac{v^i}{a(t)}\cdot\frac{\partial\rho}{\partial x^i}\right)+\rho\left\{\frac{\dot{a}(t)}{a(t)}\left(\frac{u^2}{c^2}+n\right)
			+\frac{u_t}{c^2-u}+\sum_{i=1}^{n}\left(\frac{\partial v^i}{\partial x^i}\frac{1}{a(t)}+\frac{\partial u}{\partial x^i}\frac{v^i}{a(t)(c^2-u)}\right)\right\}=0,\\
			\frac{\partial v^i}{\partial t}+\sum_{j=1}^{n}\frac{v^j}{a(t)}\frac{\partial v^i}{\partial x^j}+\frac{\dot{a}(t)}{a(t)}v^i-\frac{\dot{a}(t)}{a(t)}\frac{|v|^2}{c^2}v^i=0,\\
			t=0 : \rho(0,x^{1}, x^{2},\cdots, x^{n}) = \varepsilon\rho_{0}(x^{1}, x^{2},\cdots, x^{n}), \\
			t=0 : v(0,x^{1},x^{2},\cdots, x^{n}) = \varepsilon v_0(x^{1},x^{2},\cdots, x^{n}),
		\end{cases}
	\end{equation}
	where $u=|v|^2=\sum_{i=1}^{n}(v^i)^2$.
	\begin{Remark}
		From the first equation of \eqref{ME2}, we see that the density $\rho$ depends on the first order derivatives of the velocity $v$, thus, $\rho$ itself loses one degree of the differentiability, that is the reason that we assume that $ v_0 \in C^{2}(\mathbb{R}^{n}) $.
	\end{Remark}
	Next result is concerned with \eqref{ME2}， which is also a continuation of Theorem \ref{thm:1}.
	\begin{Theorem}\label{thm:2}
		Under the Assumptions $H_1$ and $H_2$, there exists a positive constant $\varepsilon_0$, such that the
		Cauchy problem \eqref{ME2} admits a global classical solution $(\rho, v)\in C^{1}([0,+\infty)\times \mathbb R^n)$ when $\varepsilon\in[0,\varepsilon_0]$. We can further demonstrate that $(\rho,\nabla \rho)$ are bounded for all $t\in[0,+\infty)$, and as $t$ tends to $+\infty$, they decay to zero.
		%While under the Assumptions $H_3$ or $H_4$, there exists a class of initial data such that the solution will blow up in finite time
	\end{Theorem}
	\begin{Remark}
		In \cite{Speck}, Speck has obtained the global existence of the classical solution to the relativistic dust under $\int_{0}^{\infty}a^{-2}(t)dt<+\infty$ by the energy method. The difference between this work and \cite{Speck} lies in that we use the method of the characteristics instead of energy estimates. Besides, this work gives a clear classification on the scale factor to the global existence and the blowup of the classical solution to the relativistic dust.
	\end{Remark}
	%\begin{Remark}\label{rem:1.6}
	%		Under Assumptions $H_1$
	%		and $H_2$, . However, we cannot guarantee the boundedness of $\partial\rho$, as this depends on the Hessian matrix of the characteristic lines, which proves to be quite complex to analyze.
	%	\end{Remark}
In order to study the blowup phenomenon of the density when the velocity blows up, we study the spherically symmetric solution $\left(\rho(t,r),\textbf{v}(t,r)\frac{x^i}{|x|}\right)$ with $r=|x|=\sqrt{\sum_{i=1}^{n}(x^i)^2}$ of \eqref{ME1}.
\begin{Theorem}\label{thm:3}
	Under the Assumption $H_3$ or $H_4$, and spherically symmetric initial data $\left(\rho_{0}(r), \textbf{v}_0(r)\frac{x^i}{r}\right)$, we have $(\rho,\textbf{v})\in C^1([0,+\infty)\times \mathbb R^{+})$ when $\textbf{v}_0^{'}(\alpha)\geq 0$ for all $\alpha$. On the other hand, if there exists an $\alpha$ such that $\textbf{v}_0^{'}(\alpha)<0$, then the density $\rho$ and the gradient of the velocity $\textbf{v}$ will blow up simultaneously.
\end{Theorem}
\begin{Remark}
	Based on Theorem \ref{thm:1}, under the Assumption $H_3$ or $H_4$ and $\displaystyle\inf_{\alpha}\det\left(\frac{\partial v_0}{\partial \alpha}\right)<0$, we can make sure that $\nabla v$ tends to $-\infty$, but it is difficult to tell whether the density blows up or not. While under the spherically symmetric initial data, we can obtain that $\rho$ and $|\nabla v|$ will blow up at the same time, which means that the $C^0$ norm instead of the $C^1$ one of system \eqref{ME2} will blow up in finite time.
	%This phenomenon is consistent with Majda's conjecture \cite{Majda}, since the relativistic pressureless fluid is a completely linearly degenerate system.
\end{Remark}
\begin{Remark}
	Under the Assumption $H_4$ and $n=1$, the third author in \cite{Wei1} has shown that the
	Cauchy problem \eqref{ME1} admits a global classical solution $(\rho, v)\in C^{1}([0,+\infty)\times \mathbb R)$ if and only if $v_0^{\prime}(x)\geq 0$ for arbitrary $x$. If there exists $x_0$ such that $v_0^{\prime}(x_0)< 0$, then the density must blow up in finite time. This phenomenon is consistent with a famous Majda's conjecture \cite{Majda}, which states that the system of hyperbolic conservation laws with totally linearly degenerate characteristics with initial data belong to $H^s(\mathbb R^n)$ with $s\geq 1+\frac{n}{2}$ admits a global smooth solution, unless there exists a $T^*$, such that the solution itself escapes every compact subset when $t$ tends to $T^*$. By Theorem \ref{thm:3}, we find that in the curved spacetime, the density can also blow up in finite time.
\end{Remark}
A good property of the pressureless dust is that the momentum equations do not depend on the density $\rho$, and the density and the velocity have the same propagation speeds. That is the main reason for us to deal with this system in several space variables by the characteristic method.

We arrange this paper as follows, in Section 2, we give some preparations and we study the property of the velocity and prove Theorems \ref{thm:1} and \ref{thm:s} in Section \ref{sec:2}. In Section \ref{sec:3}, we study the property of the density based on the results of Theorem \ref{thm:1}. The spherically symmetric solution is studied in Section \ref{sec:4}.

\section{Preliminaries}
We first give several lemmas, which are important for the proof of the main results.

The first lemma is the well-known Hadamard's lemma, which can
be found in \cite{Kong, Li}.
\begin{Lemma}\label{rem:2.1}
	A $C^{2}$ \textit{mapping} $x:\mathbb{R}^n\to \mathbb{R}^n,x=\varphi(\alpha)$ is a global $C^{2}$
	diffeomorphism if and only if
	\begin{itemize}
		\item $\varphi$ is a proper mapping, namely,
		$$|\varphi(\alpha)|\to\infty \quad as \quad|\alpha|\to\infty .$$
		\item $\varphi$ is a local $C^1$ diffeomorphism, i.e.,
		$$\det\left(\frac{\partial\varphi^i(\alpha)}{\partial\alpha^j}\right)\neq0,\quad\forall \alpha\in\mathbb{R}^n.$$
	\end{itemize}
\end{Lemma}
Denote the $n\times n$ matrix $A$ as
\begin{equation*}
	A:=(a_{ij})=\left(
	\begin{array}{cccc}
		a_{11}&a_{12},&\cdots &a_{1n}\\
		a_{21}&a_{22},&\cdots & a_{2n}\\
		\cdots&\cdots&\cdots&\cdots\\
		a_{n1}&a_{n2},&\cdots & a_{nn}\\
	\end{array}\right).
\end{equation*}

The following lemmas are fundamental results in linear algebra, demonstrating that an $n\times n$ strictly diagonally dominant matrix is invertible.

\begin{Lemma}\label{Lem:2.2} If $A$ is a $n\times n$ matrix and satisfies 	$
	|a_{ii}|>\sum_{j\neq i}|a_{ij}|$, $1\leqslant i\leqslant n
	$, then we can get $\det{A}\neq0$.
\end{Lemma}
\begin{proof}
	For arbitrary $x\in\mathbb{R}^n$\,($x\neq0$), there exists a $x_j\neq0$ such that $|x_j|=\sup|x_i|$.
	Set $y=Ax$ where $y_j=\sum_{i=1}^{n}a_{ji}x_i$, then we get
	\begin{equation*}
		|y_j|\geqslant|x_j||a_{jj}|-\sum_{i\neq j}|x_i||a_{ji}|\geqslant|x_j|\left(|a_{jj}|-\sum_{i\neq j}|a_{ji}|\right)>0.
	\end{equation*}
	Then  $Ax=0$ has only zero solution, we get  $\det{A}\neq0$.
	
\end{proof}
\begin{Corollary}\label{cor:2.1}
	If $A$ is a $n\times n$ matrix and satisfies 	$
	a_{ii}>\sum_{j\neq i}|a_{ij}|$, $1\leqslant i\leqslant n
	$, then we can get $\det{A}>0$.
\end{Corollary}
\begin{Lemma}\label{Lem:2.3} If $\det A\neq 0$, then the $A$ has an inverse matrix $A^{-1}$, and the entry of $A^{-1}$ located at the $i-$th row and $j-$th column is calculated by taking the cofactor $(A^{*})_{ji}$ and dividing it by $\det A$, i.e.,
	\begin{equation*}
		(A^{-1})_{ij}=\frac{(A^*)_{ji}}{\det A}.
	\end{equation*}
	
\end{Lemma}
\begin{Lemma}\label{Lem:2.4}
	A $C^{2}$ \textit{mapping} $x:\mathbb{R}^n\to \mathbb{R}^n,x=\varphi(\alpha)$ is a global $C^{2}$ diffeomorphism, and we suppose further that there exist positive constants $N,\, M$ such that
	\begin{equation*}
		\left|\det\left(\frac{\partial\varphi^i(\alpha)}{\partial\alpha^j}\right)\right|\geqslant N ,\quad \left|\frac{\partial\varphi^i(\alpha)}{\partial\alpha^j}\right|\leqslant M,\quad\left|\frac{\partial^2\varphi^i(\alpha)}{\partial\alpha^j\partial\alpha^k}\right|\leqslant M.
	\end{equation*}
	Then there exists $M_2>0$ such that
	\begin{equation*}
		\left|\frac{\partial^2\alpha^i}{\partial x^j\partial x^k}\right|\leqslant M_2.
	\end{equation*}
\end{Lemma}
\begin{proof}
	Let $A=\left(\frac{\partial\varphi^i(\alpha)}{\partial\alpha^j}\right)$, by Lemma \ref{Lem:2.3}, we get
	\begin{equation*}
		\left|\frac{\partial\alpha^i}{\partial x^j}\right|=\left|\frac{1}{\det A}\cdot(A^*)_{ji}\right|\leqslant \frac{(n-1)!M^{n-1}}{N},
	\end{equation*}
	by the chain rule, we obtain
	\begin{equation*}
		\frac{\partial^2\alpha^i}{\partial x^j\partial x^k}=-\frac{1}{(\det A)^2}\cdot(A^*)_{ji}\cdot\sum_{l=1}^{n}\frac{\partial\det A}{\partial \alpha^l}\cdot\frac{\partial\alpha^l}{\partial x^k}+\frac{1}{\det A}\cdot\sum_{l=1}^{n}\frac{\partial (A^*)_{ji}}{\partial \alpha^l}\cdot\frac{\partial\alpha^l}{\partial x^k},
	\end{equation*}
	note that $\det A$ can be written as the sum of the $n$-th power of $\frac{\partial \varphi^i}{\partial\alpha^j}$, it follows that $\frac{\partial\det A}{\partial\alpha^l}$ can be written as a sum of terms, each involving the product of the
	$(n-1)$-th power of $\frac{\partial \varphi^i}{\partial\alpha^j}$ and the second-order derivative $\frac{\partial^2\varphi^i}{\partial\alpha^j\partial\alpha^k}$. Similarly, we can obtain $\frac{\partial(A^*)_{ji}}{\partial\alpha^l}$ can be expressed as a sum of the product of the
	$(n-2)$-th power of $\frac{\partial \varphi^i}{\partial\alpha^j}$ and the second order derivative $\frac{\partial^2\varphi^i}{\partial\alpha^j\partial\alpha^k}$. Therefore, we can obtain
	\begin{equation}
		\left|\frac{\partial^2\alpha^i}{\partial x^j\partial x^k}\right|\leqslant\left[\frac{(n!)^2M^{2n-1}}{N^2}+\frac{(n-1)\cdot(n-1)!M^{n-1}}{N}\right] 	\cdot\left(\frac{(n-1)!M^{n-1}}{N}\right):=M_2.
	\end{equation}
\end{proof}
\section{Proof of Theorems \ref{thm:1} and \ref{thm:s}}\label{sec:2}

Now we reconsider	the Cauchy problem of the relativistic Burgers equations by the characteristic method, which is also studied in \cite{Huo-Wei}
\begin{equation}\label{2.1}
	\left\{\begin{aligned}
		&\frac{\partial v^i}{\partial t}+\sum_{j=1}^{n}\frac{v^j}{a(t)}\frac{\partial v^i}{\partial x^j}+\frac{\dot{a}(t)}{a(t)}v^i-\frac{\dot{a}(t)}{a(t)}\frac{|v|^2}{c^2}v^i=0,\quad i=1,2,\cdots,n,\\
		&v^i(0,x)=\varepsilon v_0^i(x).
	\end{aligned}\right.
\end{equation}

We first define the family
of characteristics by
\begin{equation}\label{2.2}
	\left\{\begin{aligned}
		&\frac{\mathrm{d} x^i(t, \alpha)}{\mathrm{d} t}=\frac{v^i(t, \alpha)}{a(t)}, \\
		&x^i(0, \alpha)=\alpha^i,
	\end{aligned}\right.\quad i=1,2,\cdots,n,
\end{equation}
where $v^{i}(t,\alpha)=v^{i}(t,x^{1}(t,\alpha),\cdots,x^{n}(t,\alpha))$.

Along the characteristics defined above, we get
\begin{equation}\label{2.3}
	\left\{\begin{aligned}
		&\frac{\mathrm{d}v^i(t,\alpha)}{\mathrm{d}t}=-\frac{\dot{a}(t)}{a(t)}v^i(t,\alpha)+\frac{\dot{a}(t)}{a(t)}\frac{|v(t,\alpha)|^2}{c^2}v^i(t,\alpha),\\
		&v^i(0,\alpha)=\varepsilon v_0^i(\alpha),
	\end{aligned}\right.\quad i=1,2,\cdots,n.
\end{equation}

Let $u(t,\alpha)=\sum_{i=1}^{n}|v^{i}(t,\alpha)|^2$, then	\begin{equation}\label{2.4}
	\left\{\begin{aligned}
		&\frac{\mathrm{d}u(t,\alpha)}{\mathrm{d}t}=-2\frac{\dot{a}(t)}{a(t)}u(t,\alpha)+2\frac{\dot{a}(t)}{a(t)}\frac{u^2(t,\alpha)}{c^2}, \\
		&u(0,\alpha)=\varepsilon^2|v_0(\alpha)|^2.
	\end{aligned}\right.
\end{equation}

It is easy to obtain from \eqref{2.4}
\begin{equation}\label{|v|^2}
	\frac{u(t,\alpha)}{c^2-u(t,\alpha)}=\frac{f_0^2(\alpha)}{a^2(t)},
\end{equation}
where $ f_0(\alpha)=\frac{\varepsilon |v_0(\alpha)|}{\sqrt{c^2-\varepsilon^2|v_0(\alpha)|^2}}$,
then
\begin{equation}\label{u}
	|v(t,\alpha)|^2=u(t,\alpha)=\frac{c^2f_0^2(\alpha)a^{-2}(t)}{1+f_0^2(\alpha)a^{-2}(t)}.
\end{equation}

Inserting \eqref{u} into \eqref{2.3}, we get
\begin{equation}\label{3.6}
	v^i(t,\alpha)=\frac{cg_0^i(\alpha)a^{-1}(t)}{\sqrt{1+f_0^2(\alpha)a^{-2}(t)}},\quad i=1,2,\cdots,n，
\end{equation}
where $ g_0^i(\alpha)=\frac{\varepsilon v_0^i(\alpha)}{\sqrt{c^2-\varepsilon^2|v_0(\alpha)|^2}}.$

According to \eqref{2.2}, we have
\begin{equation}\label{2.5}
	x^i(t,\alpha)=\alpha^i+cg_0^i(\alpha)\int_0^t\frac{a^{-2}(s)}{\sqrt{1+f_0^2(\alpha)a^{-2}(s)}}\mathrm{d}s.
\end{equation}
By \eqref{3.6} and \eqref{2.5}, we obtain
\begin{equation}\label{3.8}
	\frac{\partial v^i}{\partial x^j}=\sum_{l=1}^{n}\frac{\partial v^i}{\partial \alpha^l}\frac{\partial \alpha^{l} }{\partial x^j}.
\end{equation}
%	Now we are ready to prove Theorem \ref{thm:1}.
In the following, in order to show that $v(t,x)\in [0,T(\varepsilon))\times \mathbb R^n)$, we need to prove $v^i$ and $\frac{\partial v^i}{\partial x^j}$ are bounded in the existence domain of the classical solution.

\subsection{Proof of Theorem \ref{thm:1}}
\begin{proof}
	Since $v_0$ is a bounded $C^2$ function, we can assume that there exists a positive constant $N_0>0$ such that
	\begin{equation}\label{f_0estimate}
		\|v_0\|_{C^2(\mathbb{R}^n)}\leqslant N_0.
	\end{equation}
	Here, $ \lVert \varPhi\rVert _{ C^{2}(\mathbb R^n)} := \lVert \varPhi\rVert _{ C^{0}(\mathbb R^n)} + \lVert \partial\varPhi\rVert _{ C^{0}(\mathbb R^n)} + \lVert \partial^{2}\varPhi\rVert _{ C^{0}(\mathbb R^n)} $ and $\|\varPhi(\cdot)\|_{C^{0}(\mathbb R^n)}:=\sup_{x\in \mathbb R^n}|\varPhi(x)|$.
	
	Let $	M_0=n\sqrt{n}N_0,
	$ then direct calculations give
	\begin{align}
		\|f_0\|_{C^0(\mathbb{R}^n)}	&\leqslant\frac{\varepsilon M_0}{\sqrt{c^2-\varepsilon^2M_0^2}},\\
		\Biggl|\frac{\partial f_0}{\partial\alpha^j}\Biggl|	&=\left|\frac{c^2\varepsilon\frac{\partial|v_0|}{\partial\alpha^j}}{(c^2-\varepsilon^2|v_0|^2)^{3/2}}\right|\leqslant\frac{\sqrt{n}c^2\varepsilon N_0}{(c^2-\varepsilon^2M_0^2)^{3/2}},\\
		\Big\|\frac{\partial f_0}{\partial\alpha}\Big\|_{C^0(\mathbb{R}^n)}	&=\sum_{j=1}^{n}\Biggl|\frac{\partial f_0}{\partial\alpha^j}\Biggl|\leqslant\frac{n\sqrt{n}c^2\varepsilon N_0}{(c^2-\varepsilon^2M_0^2)^{3/2}}=\frac{c^2\varepsilon M_0}{(c^2-\varepsilon^2M_0^2)^{3/2}} ,\\
		\|g_0\|_{C^0(\mathbb{R}^n)}&\leqslant\frac{\varepsilon M_0}{\sqrt{c^2-\varepsilon^2M_0^2}} ,\\
		\Biggl|\frac{\partial g_0^i}{\partial\alpha^j}\Biggl|&=\left|\frac{c^2\varepsilon \frac{\partial v_0^i}{\partial\alpha^j}+\varepsilon^3|v_0|(v_0^i\frac{\partial|v_0|}{\partial\alpha^j}-|v_0|\frac{\partial v_0^i}{\partial\alpha^j})}{(c^2-\varepsilon^2|v_0|^2)^{3/2}}\right|\leqslant\frac{c^2\varepsilon N_0+2\sqrt{n}\varepsilon^3N_0^3}{(c^2-\varepsilon^2|v_0|^2)^{3/2}} ,\\
		\Big\|\frac{\partial g_0^i}{\partial\alpha}\Big\|_{C^0(\mathbb{R}^n)}&=\sum_{j=1}^{n}\Biggl|\frac{\partial g_0^i(\alpha)}{\partial\alpha^j}\Biggl|\leqslant\left|\frac{nc^2\varepsilon N_0+2n\sqrt{n}\varepsilon^3N_0^3}{(c^2-\varepsilon^2|v_0|^2)^{3/2}}\right|\leqslant\frac{c^2\varepsilon M_0+2\varepsilon^3M_0^3}{(c^2-\varepsilon^2|v_0|^2)^{3/2}}.\label{g_0/alphaestimate}
	\end{align}
	The above estimates can guarantee that $|v|$ and $\frac{\partial v^i}{\partial \alpha^{l}}$ are bounded.
	
	Next, by \eqref{2.5}, we can show that $\frac{\partial \alpha^l}{\partial x^j}$ is also bounded.
	
	Let
	\begin{equation*}
		h^i(t,\alpha)=cg_0^i(\alpha)\int_0^t\frac{a^{-2}(s)}{\sqrt{1+f_0^2(\alpha)a^{-2}(s)}}\mathrm{d}s,
	\end{equation*}
	then we have
	\begin{equation}\label{matrix}
		\frac{\partial x^i}{\partial \alpha^j}=\delta^{ij}+\frac{\partial h^{i}(t,\alpha)}{\partial \alpha^j},
	\end{equation}
	where $\delta^{ij}=1$ if $i=j$, $\delta^{ij}=0$ if $i\neq j$, and
	\begin{align}\label{2.8}
		\frac{\partial h^i}{\partial \alpha^j}=&c\frac{\partial g_0^i(\alpha)}{\partial\alpha^j}\int_0^t\frac{a^{-2}(s)}{(1+f_0^2(\alpha)a^{-2}(s))^{3/2}}\mathrm{d}s \nonumber\\
		+&cf_0(\alpha)\Big[f_0(\alpha)\frac{\partial g_0^i(\alpha)}{\partial\alpha^j}-g_0^i(\alpha)\frac{\partial f_0(\alpha)}{\partial\alpha^j}\Big]\int_0^t\frac{a^{-4}(s)}{(1+f_0^2(\alpha)a^{-2}(s))^{3/2}}\mathrm{d}s,
	\end{align}
	for $t\in[0,T(\varepsilon))$, we can obtain
	\begin{align}
		\|h(t,\cdot)\|_{C^0(\mathbb{R}^n)}&\leqslant\frac{c\varepsilon M_0}{\sqrt{c^2-\varepsilon^2 M_0^2}}\int_{0}^{t}a^{-2}(s)\mathrm{d}s ,\label{2.9}\\
		\Biggl|\frac{\partial h^i(t,\cdot)}{\partial\alpha^j}\Biggl|&\leqslant\frac{2c(c^2\varepsilon M_0+\varepsilon^3M_0^3)}{(c^2-\varepsilon^2 M_0^2)^{3/2}}\int_{0}^{t}a^{-2}(s)\mathrm{d}s+\frac{2c\varepsilon^2M_0^2(c^2\varepsilon M_0+\varepsilon^3M_0^3)}{(c^2-\varepsilon^2 M_0^2)^{5/2}}\int_{0}^{t}a^{-4}(s)\mathrm{d}s. \label{2.10}
	\end{align}
	
	In the following, we split into three cases according to Assumptions $H_1$-$H_3$ to investigate the properties of the characteristics $x(t,\alpha)$ and $v$. When $a(t)$ satisfies Assumption $H_4$, the classical Burgers equation has been studied clearly, see \cite{Li}.
	
	\textbf{Case I.}\quad Under Assumption $H_1$: $\dot{a}(t)>0$ and $\ddot{a}(t)\geqslant0$,
	it holds that
	\begin{align}\label{H1inta}
		a(t)&=1+\int_{0}^{t}\dot{a}(s)\mathrm{d}s\geqslant1+\dot{a}(0)t,\\
		\int_{0}^{t}a^{-4}(s)\mathrm{d}s&\leqslant\int_{0}^{t}a^{-2}(s)\mathrm{d}s\leqslant\int_{0}^{\infty}[1+\dot{a}(0)s]^{-2}\mathrm{d}s=\frac{1}{\dot{a}(0)}.
	\end{align}
	
	According to \eqref{2.10}, we have
	\begin{align*}
		\Biggl|\frac{\partial h^i(t,\cdot)}{\partial\alpha^j}\Biggl|\leqslant&\frac{2c(c^2\varepsilon M_0+\varepsilon^3M_0^3)}{(c^2-\varepsilon^2 M_0^2)^{3/2}}\int_{0}^{t}a^{-2}(s)\mathrm{d}s+\frac{2c\varepsilon^2M_0^2(c^2\varepsilon M_0+\varepsilon^3M_0^3)}{(c^2-\varepsilon^2 M_0^2)^{5/2}}\int_{0}^{t}a^{-4}(s)\mathrm{d}s \\
		\leqslant&\frac{1}{\dot{a}(0)}\frac{2c\varepsilon M_0(c^4+\varepsilon^4M_0^4+2c^2\varepsilon^2 M_0^2)}{(c^2-\varepsilon^2 M_0^2)^{5/2}} \\
		\leqslant&\frac{1}{2n},
	\end{align*}
	provided
	\begin{equation*}
		\varepsilon\leqslant\varepsilon_1=\min\left\{\frac{c\sqrt{1-\delta^{2/5}}}{M_0},\eta_1,1\right\},
	\end{equation*}
	where $0<\delta<1$ is a constant, and $\eta_1 M_0+\frac{\eta_1^3M_0^3}{c^2}=\frac{\dot{a}(0)}{4n}\frac{\delta c^2}{M_0^2+c^2}$.
	
	For convenience of notations, we introduce
	$$
	\Omega(c,\varepsilon,M_0):=\frac{2c\varepsilon M_0(c^4+\varepsilon^4M_0^4+2c^2\varepsilon^2 M_0^2)}{(c^2-\varepsilon^2 M_0^2)^{5/2}}.
	$$
	
	By \eqref{2.9}, we have
	\begin{align*}
		|h(t,\alpha)|\leqslant&\frac{c\varepsilon M_0}{\sqrt{c^2-\varepsilon^2 M_0^2}}\int_{0}^{t}a^{-2}(s)\mathrm{d}s \\
		\leqslant&\frac{c\varepsilon M_0}{\sqrt{c^2-\varepsilon^2 M_0^2}}\\
		\leqslant&\frac{1}{4n}\delta^{4/5},
	\end{align*}
	Thus, in this case, $|\frac{\partial x^i}{\partial \alpha^l}|$ is bounded and the matrix \eqref{matrix} is a diagonally dominant matrix, this leads to $\det\left(\frac{\partial x}{\partial \alpha}\right)>0$ by Corollary \ref{cor:2.1}. Combining Lemma \ref{Lem:2.3}, we can easily see that $|\frac{\partial \alpha^l}{\partial x^j}|$ is bounded.
	
	Thus, in this case, we have $T(\varepsilon)=\infty$.
	
	\textbf{Case II.} Under the Assumption $H_2$:  $\dot{a}(t)>0,\,\ddot{a}(t)<0$, and $0<\delta_0<1$ such that $a(t)\geqslant(1+t)^{\frac{1+\delta_0}{2}}$.
	We have
	\begin{align}\label{H2inta}
		\int_{0}^{t}a^{-4}(s)\mathrm{d}s\leqslant\int_{0}^{t}a^{-2}(s)\mathrm{d}s\leqslant&\int_{0}^{\infty}a^{-2}(t)\mathrm{d}t \nonumber\\
		\leqslant&\int_{0}^{\infty}(1+t)^{-1-\delta_0}\mathrm{d}t=\delta_0^{-1}.
	\end{align}
	
	By \eqref{2.10}, we get
	\begin{align*}
		\Biggl|\frac{\partial h^i(t,\cdot)}{\partial\alpha^j}\Biggl|\leqslant&\frac{2c(c^2\varepsilon M_0+\varepsilon^3M_0^3)}{(c^2-\varepsilon^2 M_0^2)^{3/2}}\int_{0}^{t}a^{-2}(s)\mathrm{d}s+\frac{2c\varepsilon^2M_0^2(c^2\varepsilon M_0+\varepsilon^3M_0^3)}{(c^2-\varepsilon^2 M_0^2)^{5/2}}\int_{0}^{t}a^{-4}(s)\mathrm{d}s \\
		\leqslant&\frac{1}{\delta_0}\Omega(c,\varepsilon,M_0) \\
		\leqslant&\frac{1}{2n},
	\end{align*}
	when
	\begin{equation*}
		\varepsilon\leqslant\varepsilon_2=\min\left\{\frac{c\sqrt{1-\delta^{2/5}}}{M_0},\eta_2,1\right\},
	\end{equation*}
	where $0<\delta<1$ is a constant, and $\eta_2 M_0+\frac{\eta_2^3M_0^3}{c^2}=\frac{\delta_0}{2n}\frac{c^2}{M_0^2+c^2}$.
	
	According to \eqref{2.9}, we have
	\begin{align*}
		|h(t,\alpha)|\leqslant&\frac{c\varepsilon M_0}{\sqrt{c^2-\varepsilon^2 M_0^2}}\int_{0}^{t}a^{-2}(s)\mathrm{d}s \\
		\leqslant&\frac{1}{\delta_0}\cdot\frac{c\varepsilon M_0}{\sqrt{c^2-\varepsilon^2 M_0^2}} \\
		\leqslant&\frac{1}{4n}\delta^{4/5}.
	\end{align*}
	
	Similar to \textbf{Case I}, we can obtain $T(\varepsilon)=\infty$.
	
	Combining \textbf{Case I} and \textbf{Case II}, we can conclude that when $a(t)$ satisfies Assumption $H_{1}$ or $H_{2}$, the Cauchy problems \eqref{equation} have a global classical solution $v(t, x)\in C^{1}([0,+\infty)\times \mathbb R^n)$.
	
	\textbf{Case III.} Under Assumption $H_3$, i.e., $a(t)\sim(1+t)^l$, $0< l\leqslant\frac{1}{2}$, we have
	\begin{align*}
		\int_{0}^{t}a^{-2}(s)\mathrm{d}s=\int_{0}^{t}(1+s)^{-2l}\mathrm{d}s=\begin{cases}
			\ln(1+t),\quad l=\frac{1}{2}, \\
			\frac{(1+t)^{1-2l}-1}{1-2l},\quad 0< l<\frac{1}{2},
		\end{cases}
	\end{align*}

	Then by \eqref{2.10}, we can obtain
	\begin{align*}
		\left|\frac{\partial h^i(t,\alpha)}{\partial\alpha^j}\right|\leqslant\begin{cases}
			\Omega(c,\varepsilon,M_0)\ln(1+t),\quad l=\frac{1}{2}, \\
			\Omega(c,\varepsilon,M_0)\cdot\frac{(1+t)^{1-2l}-1}{1-2l},\quad 0< l<\frac{1}{2},
		\end{cases}
	\end{align*}
	if we suppose $\Big|\frac{\partial h^i(t,\alpha)}{\partial\alpha^j}\Big|\leq \frac{1}{2n}$, then we have
	\begin{equation*}
		t_*(\varepsilon)=\left\{\begin{aligned}
			&\exp\left(\frac{1}{n\Omega(c,\varepsilon,M_0)}\right)-1,\quad l=\frac{1}{2}, \\
			&\left(\frac{1}{n\Omega(c,\varepsilon,M_0)}+1\right)^{1/(1-2l)}-1,\quad 0< l<\frac{1}{2}.
		\end{aligned}\right.
	\end{equation*}
	This is guaranteed provided that
	$$\varepsilon\leqslant\varepsilon_3:=\frac{c\sqrt{1-\delta^{2/5}}}{M_0}.$$
	
	Therefore, the life span $T(\varepsilon)$ of the classical solution satisfies
	\begin{equation}
		T(\varepsilon)(\leqslant t_*(\varepsilon))=\left\{\begin{aligned}
			&\exp\left(\frac{\delta c^2}{4\varepsilon nM_0(c^2+M_0^2)}\right)-1,\quad l=\frac{1}{2}, \\
			&\left(\frac{(1-2l)\delta c^2}{4\varepsilon nM_0(c^2+M_0^2)}+1\right)^{1/(1-2l)}-1,\quad 0< l<\frac{1}{2}.
		\end{aligned}\right.
	\end{equation}
	\begin{Remark}
		The main idea for us to get the life span of the classical solution is to ensure the determinant of $\Big|\frac{\partial x}{\partial \alpha}(t,\alpha)\Big|>0$ for all $t\in[0,T(\epsilon)]$ based on Lemmas \ref{Lem:2.2} and \ref{Lem:2.3}.
	\end{Remark}
	%\subsection{Blow up criterion}
	Next, we focus on the blowup criterion of the solution according to the dimensions.
	
	When $n=1$,  we have $f_0=g_0^i$, and
	\begin{equation*}
		\frac{\partial x(t,\alpha)}{\partial \alpha}=1+cf_0'(\alpha)\int_{0}^{t}\frac{a^{-2}(s)}{(1+f_0^2(\alpha)a^{-2}(s))^{3/2}}\mathrm{d}s,
	\end{equation*}
	where
	\begin{equation*}
		f_0'(\alpha)=\frac{c^2\varepsilon v_0'(\alpha)}{(c^2-\varepsilon v_0^2(\alpha))^{3/2}}.
	\end{equation*}
	
	If $v_0'(\alpha)\geqslant0,$ then $	\frac{\partial x(t,\alpha)}{\partial \alpha}\geqslant1$, we can obtain that  $T(\varepsilon)=\infty$.
	
	If  $\displaystyle{\inf_{\alpha}}\, v_0'(\alpha)<0$, then there exists an $\alpha$ such that $v_0'(\alpha)<0$, and \begin{equation*}
		\frac{\partial x(t,\alpha)}{\partial \alpha}\leqslant1+\frac{cf_0'(\alpha)}{(1+f_0^2(\alpha))^{3/2}}\int_{0}^{t}a^{-2}(s)\mathrm{d}s\leqslant1+\frac{\varepsilon v_0'(\alpha)}{(1+f_0^2(\alpha))^{3/2}}\int_{0}^{t}a^{-2}(s)\mathrm{d}s.
	\end{equation*}
	
	By $a(t)\sim(1+t)^l$, we have
	\begin{equation*}
		\frac{\partial x(t,\alpha)}{\partial \alpha}\leqslant\left\{\begin{aligned}
			&1+\frac{\varepsilon v_0'(\alpha)}{(1+f_0^2(\alpha))^{3/2}}\cdot\ln(1+t),\quad l=\frac{1}{2}, \\
			&1+\frac{\varepsilon v_0'(\alpha)}{(1+f_0^2(\alpha))^{3/2}}\cdot\frac{(1+t)^{1-2l}-1}{1-2l},\quad 0< l<\frac{1}{2},
		\end{aligned}\right.
	\end{equation*}
	and we can also obtain $	\frac{\partial x(t,\alpha)}{\partial \alpha}$ tends to $0$ when $t$ tends to $t_1(\varepsilon)$, which is defined by
	\begin{equation}
		t_1(\varepsilon)=\left\{\begin{aligned}
			&\exp\left(-\frac{(1+f_0^2(\alpha))^{3/2}}{\varepsilon v_0'(\alpha)}\right)-1,\quad l=\frac{1}{2}, \\
			&\left(-\frac{(1-2l)(1+f_0^2(\alpha))^{3/2}}{\varepsilon v_0'(\alpha)}+1\right)^{1/(1-2l)}-1,\quad 0< l<\frac{1}{2}.
		\end{aligned}\right.
	\end{equation}
	
	%The derivative of the solution will blow up when  $t\geqslant t_1(\varepsilon)$ and the upper bound of the solution is consistent with the lower bound of the lifespan of $ v $, which means that the bound of the lifespan is sharp.
	
	When $n>1$, let $I$ be the identity matrix, then
	\begin{align}\label{3.24}
		\frac{\partial x}{\partial \alpha}&=I+\left(\frac{\partial v_0}{\partial \alpha}\right)\frac{c^3\varepsilon}{[c^2-\varepsilon^2|v_0|^2]^{3/2}}\int_{0}^{t}\frac{a^{-2}(s)}{[1+f_0^2(\alpha)a^{-2}(s)]^{3/2}}\mathrm{d}s \nonumber\\
		&-(|v_0|^2I-v_0^Tv_0)\left(\frac{\partial v_0}{\partial \alpha}\right)\frac{c\varepsilon^3}{[c^2-\varepsilon^2|v_0|^2]^{3/2}}\int_{0}^{t}\frac{a^{-2}(s)-a^{-4}(s)}{[1+f_0^2(\alpha)a^{-2}(s)]^{3/2}}\mathrm{d}s,
	\end{align}
	this expression is directly derived from equations \eqref{matrix} and \eqref{2.8}.
	
	By Hadamard's Lemma \ref{rem:2.1}, we can see that the system has a global classical condition if and only if $\det\left(\frac{\partial x}{\partial \alpha}\right)\neq 0$. Since when $t=0$, $\det\left(\frac{\partial x}{\partial \alpha}\right)=1$, by the continuity of $\det\left(\frac{\partial x}{\partial \alpha}\right)$, $\det\left(\frac{\partial x}{\partial \alpha}\right)\neq 0$ is equivalent to $\det\left(\frac{\partial x}{\partial \alpha}\right)> 0$ for all $t>0$ and $\alpha\in \mathbb R^n$.
	
	By the definition of $\nabla v$, the solution must blow up if and only if there exists an $\alpha'\in\mathbb{R}^n$ such that $\displaystyle\det\left(\frac{\partial x}{\partial\alpha}(t',\alpha')\right)=0$ for $0<t'<\infty$.
\end{proof}
%	\subsection{Proof of Remark \ref{rem:1.4}}

\subsection{Proof of Theorem \ref{thm:s}.}
In the following, we show that $\displaystyle\inf_{\alpha}\det\left(\frac{\partial v_0}{\partial \alpha}\right)<0$ can ensure the blowup of $|\nabla v|$ in finite time.
\begin{proof}
	%We first prove the solution will blow up if $\displaystyle\inf_{\alpha}\det\left(\frac{\partial v_0}{\partial \alpha}\right)<0$ .
	
	Let
	\begin{equation*}
		K^{ij}(\alpha)=\left(\frac{\partial v_0^i}{\partial\alpha^j}\right)|v_0|^2-v_0^i\sum_{k=1}^{n}v_0^k\frac{\partial v_0^k}{\partial\alpha^j},
	\end{equation*}
	then we have
	\begin{align*}
		\frac{\partial h^i(t,\alpha)}{\partial\alpha^j}&=c\varepsilon\frac{c^2\frac{\partial v_0^i}{\partial\alpha^j}-\varepsilon^2K^{ij}}{[c^2-\varepsilon^2|v_0|^2]^{3/2}}\int_{0}^{t}\frac{a^{-2}(s)}{[1+f_0^2(\alpha)a^{-2}(s)]^{3/2}}\mathrm{d}s \nonumber\\
		&+\frac{c\varepsilon^3K^{ij}}{[c^2-\varepsilon^2|v_0|^2]^{3/2}}\int_{0}^{t}\frac{a^{-4}(s)}{[1+f_0^2(\alpha)a^{-2}(s)]^{3/2}}\mathrm{d}s.
	\end{align*}
	
	Denote
	\begin{align*}
		&D^{ij}(\alpha)=\frac{c\varepsilon}{[c^2-\varepsilon^2|v_0|^2]^{3/2}}\cdot\left(c^2\frac{\partial v_0^i}{\partial\alpha^j}-\varepsilon^2K^{ij}\right), \\
		&E^{ij}(\alpha)=\frac{c\varepsilon^3K^{ij}}{[c^2-\varepsilon^2|v_0|^2]^{3/2}}, \\
		&F_2(t,\alpha)=\int_{0}^{t}\frac{a^{-2}(s)}{[1+f_0^2(\alpha)a^{-2}(s)]^{3/2}}\mathrm{d}s, \\
		&F_4(t,\alpha)=\int_{0}^{t}\frac{a^{-4}(s)}{[1+f_0^2(\alpha)a^{-2}(s)]^{3/2}}\mathrm{d}s.
	\end{align*}
	
	Then, we can get
	\begin{equation*}
		I+\left(\frac{\partial h^i(t,\alpha)}{\partial\alpha^j}\right)=\delta^{ij}+D^{ij}(\alpha)F_2(t,\alpha)+E^{ij}(\alpha)F_4(t,\alpha).
	\end{equation*}
	The above explicit expression indicates that the determinant of the matrix $I+\left(\frac{\partial h^i(t,\alpha)}{\partial\alpha^j}\right)$  is a polynomial in
	$F_2(t,\alpha)$ and $F_4(t,\alpha)$
	\begin{equation*}
		\det\left(I+\left(\frac{\partial h^i(t,\alpha)}{\partial\alpha^j}\right)\right)=\sum_{k+s\leqslant n}G^{ks}(\alpha)F_2^k(t,\alpha)F_4^s(t,\alpha),
	\end{equation*}
	where the coefficients $G^{ks}(\alpha)$ is the combination of $D_j^i(\alpha)$ and $E_j^i(\alpha)$, and depend solely on $\alpha$. $k,s$ denotes the degree of $F_2(t,\alpha)$ and $F_{4}(t,\alpha)$, respectively.
	
	It is easy to get
	\begin{equation*}
		\lim_{t\to\infty}\frac{F_4(t,\alpha)}{F_2(t,\alpha)}=\lim_{t\to\infty}\frac{\frac{a^{-4}(t)}{[1+f_0^2(\alpha)a^{-2}(t)]^{3/2}}}{\frac{a^{-2}(t)}{[1+f_0^2(\alpha)a^{-2}(t)]^{3/2}}}=\lim_{t\to\infty}a^{-2}(t)=0.
	\end{equation*}
	
	Then
	\begin{align*}
		\lim_{t\to\infty}\det\left(I+\left(\frac{\partial h^i(t,\alpha)}{\partial\alpha^j}\right)\right)&=\lim_{t\to\infty}F_2^n(t,\alpha)\sum_{k+s\leqslant n}G^{ks}(\alpha)\frac{F_2^k(t,\alpha)F_4^s(t,\alpha)}{F_2^n(t,\alpha)} \nonumber\\
		&=\lim_{t\to\infty}G^{n0}(\alpha)F_2^n(t,\alpha)=-\infty.
	\end{align*}
	Here $$
	G^{n0}(\alpha)=\frac{(c\varepsilon)^n}{[c^2-\varepsilon^2|v_0|^2]^{3n/2}}\sum_{\sigma\in S_n}\mathrm{sgn}(\sigma)\left(\prod_{i=1}^n\left[c^2\frac{\partial v_0^i}{\partial\alpha^{\sigma(i)}}-\varepsilon^2K^{i,\sigma(i)}\right]\right),$$
	where $S_n$ is the set of all permutations of $\{1,2,\cdots,n\}$, $\mathrm{sgn}(\sigma)$ is the sign of the permutation, which is positive for even permutations and negative for odd permutations, and $\sigma(i)$ is the i-th element of the permutation $\sigma$.
	
	If there exists an $\alpha$ such that $\det\left(\frac{\partial v_0}{\partial \alpha}\right)<0$, we have
	\begin{equation*}
		\lim_{\varepsilon\rightarrow 0}\sum_{\sigma\in S_n}\mathrm{sgn}(\sigma)\left(\prod_{i=1}^n\left[c^2\frac{\partial v_0^i}{\partial\alpha^{\sigma(i)}}-\varepsilon^2K^{i,\sigma(i)}\right]\right)=c^2\det\left(\frac{\partial v_0}{\partial \alpha}\right)<0.
	\end{equation*}
	Thus, we can find an $\varepsilon_3$ such that when $\varepsilon\leq \varepsilon_3$, there holds $G^{n0}(\alpha)<0$ for arbitrary $t$, and
	\begin{equation*}
		\lim_{t\to\infty}\det\left(I+\left(\frac{\partial h^i(t,\alpha)}{\partial\alpha^j}\right)\right)=-\infty.
	\end{equation*}
	On the other hand, when $t=0$, it holds that
	\begin{equation*}
		\det\left(I+\left(\frac{\partial h^i(t,\alpha)}{\partial\alpha^j}\right)\right)=1.
	\end{equation*}
	
	So there must exist a $t_2<+\infty$ such that $$	\det\left(I+\left(\frac{\partial h^i(t_2,\alpha)}{\partial\alpha^j}\right)\right)=0.$$
	By the blowup criterion in Theorem \ref{thm:1}, $\nabla v$ blows up when $t$ tends to $t_2$.
\end{proof}
\subsection{Examples for Remark \ref{rem:1.4}}
When $\displaystyle\inf_{\alpha}\det\left(\frac{\partial v_0}{\partial \alpha}\right)\geqslant0$, we can find examples such that the solution admits a global solution or blows up in finite time.

Let $n=2$ and $a(t)=(1+t)^{\frac{1}{2}}$, then $v_0(\alpha)=(v_0^1(\alpha^1,\alpha^2),v_0^2(\alpha^1,\alpha^2))$ and \eqref{3.24} turns into
\begin{equation}\label{3.25}
	\frac{\partial x}{\partial \alpha}=I+\left(\frac{\partial v_0}{\partial \alpha}\right)Y_{1}(t,\alpha)\Phi(t,\alpha)-\left(|v_0|^2I-v_0^Tv_0\right)\left(\frac{\partial v_0}{\partial \alpha}\right)Y_{2}(t,\alpha)\Psi(t,\alpha),
\end{equation}
where
\begin{align*}
	&	Y_{1}=\frac{c^3\varepsilon }{(c^2-\varepsilon^2|v_0|^2)^{\frac{3}{2}}}>0,\\
	&	Y_{2}=\frac{c\varepsilon^3 }{(c^2-\varepsilon^2|v_0|^2)^{\frac{3}{2}}}>0,\\
	&	\Phi(t,\alpha)=\int_{0}^{t}\frac{(1+s)^{-1}}{[1+f_0^2(\alpha)(1+s)^{-1}]^{3/2}}\mathrm{d}s>0,\\
	&	\Psi(t,\alpha)=\int_{0}^{t}\frac{(1+s)^{-1}-(1+s)^{-2}}{[1+f_0^2(\alpha)(1+s)^{-1}]^{3/2}}\mathrm{d}s>0.
\end{align*}

Set $ \delta\geqslant0 $ and let
\begin{align*}
	v_0^1 =\delta\alpha^1+ \arctan\alpha^1,	\,
	v_0^2 = \delta\alpha^2+\arctan\alpha^2,
\end{align*}
then for arbitrary $\alpha$,
\begin{equation*}
	\displaystyle\inf_{\alpha}\det\left(\frac{\partial v_0}{\partial \alpha}\right)\geqslant0.
\end{equation*}

%, and
%\begin{equation*}
%	\displaystyle\inf_{\alpha}\det\left(\frac{\partial v_0}{\partial \alpha}\right)>0,
%	\end{equation*}
%	if $\delta=0$.

Then
\begin{eqnarray*}
\det\left(\frac{\partial x}{\partial \alpha}\right)&=&\Big[1+(Y_1\Phi(t,\alpha)-(v_0^2)^2Y_2\Psi(t,\alpha))\left(\delta+\frac{1}{1+(\alpha^1)^2}\right)\\
&+&\left(Y_1\Phi(t,\alpha)-(v_0^1)^2Y_2\Psi(t,\alpha)\right)\left(\delta+\frac{1}{1+(\alpha^2)^2}\right)
\\ &+&\left((Y_1\Phi(t,\alpha))^2-Y_1Y_2\Phi(t,\alpha)\Psi(t,\alpha)((v_0^1)^2+(v_0^2)^2)\right)\\
&\times&\left(\delta+\frac{1}{1+(\alpha^1)^2}\right)\left(\delta+\frac{1}{1+(\alpha^2)^2}\right)\Big]>0.
\end{eqnarray*}

In this case, $v$ never blows up in finite time.

However, if we set
\begin{align*}
v_0^1 =-\delta\alpha^1- \arctan\alpha^1,	\,
v_0^2 = -\delta\alpha^2-\arctan\alpha^2,
\end{align*}
then for arbitrary $\alpha$,
\begin{equation*}
\displaystyle\inf_{\alpha}\det\left(\frac{\partial v_0}{\partial \alpha}\right)\geqslant0,
\end{equation*}
and
\begin{eqnarray*}
\det\left(\frac{\partial x}{\partial \alpha}\right)&=&1-(Y_1\Phi(t,\alpha)+(v_0^2)^2Y_2\Psi(t,\alpha))\left(\delta+\frac{1}{1+(\alpha^1)^2}\right)\\
&-&\left(Y_1\Phi(t,\alpha)+(v_0^1)^2Y_2\Psi(t,\alpha)\right)\left(\delta+\frac{1}{1+(\alpha^2)^2}\right)
\\ &+&\left((Y_1\Phi(t,\alpha))^2+Y_1Y_2\Phi(t,\alpha)\Psi(t,\alpha)((v_0^1)^2+(v_0^2)^2)\right)\\
&\times&\left(\delta+\frac{1}{1+(\alpha^1)^2}\right)\left(\delta+\frac{1}{1+(\alpha^2)^2}\right).
\end{eqnarray*}

In this case, when $\alpha =0$,
\begin{align*}
\det\left(\frac{\partial x}{\partial \alpha}\right)=1-2\varepsilon\Phi(t)(\delta+1)
+(\varepsilon\Phi(t)(\delta+1))^2,
\end{align*}
and
\begin{equation*}
\Phi(t)=\int_{0}^{t}(1+s)^{-1}\mathrm{d}s.
\end{equation*}

Thus, when $t=t^{*}=e^{\frac{1}{\varepsilon(\delta+1)}}-1$, $ \Phi(t^{*})=\frac{1}{\varepsilon(\delta+1)} $, we get
\begin{equation*}
\det\left(\frac{\partial x}{\partial \alpha}\right)=0,
\end{equation*}
then we can see that $v$ blows up when $t$ tends to $t^{*}$.
%\end{Remark}

%	\begin{Remark}
%		When $n=1$, we can accurately obtain the blow up time of the solution by choosing appropriate initial data.

%		Let $v_0(x) = x^2$ with $ x \in [-1,1]$, when $x\in[-1,0)$, we have $v_0'(x)<0$. we further let $a(t)=(1+t)^{l}$ with $l=\frac{1}{4}$,  we can draw the characteristic surface of the solution  as follow:
%		\begin{figure}[h]
	%			\centering
	%			\includegraphics[width=0.7\linewidth]{
		%				characteristic.png}
	%			\caption{characteristic surface of the solution.}
	%			\label{fig:1}
	%		\end{figure}

%		The red line represents the  characteristic line arising from $\alpha=-1$.
%		\begin{figure}[h]
	%			\centering
	%			\includegraphics[width=0.7\linewidth]{ blowup.png}
	%			\caption{The blow up point of $v$.}
	%			\label{fig:2}
	%		\end{figure}

%			It's easy to see that $v =-\infty$ when $x=-1$, which is consistent with our theoretical results.
%	\end{Remark}
%\begin{Remark} When $ a(t) = 1 $, the Cauchy problem of the  relativistic Burgers equation turns to
%\begin{equation}\label{2.14}
%\begin{cases}
%	\displaystyle\frac{\partial v^i}{\partial t}+\sum_{j=1}^{n}v^j\frac{\partial v^i}{\partial x^j}=0,\quad i=1,2,\cdots,n\\v^{i}(0,x)=\epsilon v_{0}^{i}(x),
%\end{cases}
%	\end{equation}
%\end{Remark}

\section{Proof of Theorem \ref{thm:2}
} \label{sec:3}
In this section, based on the results of Theorem \ref{thm:1}, we consider the classical solution of $ \rho $. We assume that $\rho>0$ in the existence domain of the classical solution, then the first equation of \eqref{ME2} can be rewritten as
%\begin{align}\label{3.22}
%	\frac{\partial\ln\rho}{\partial t}&+\sum_{i=1}^{n}\left(\frac{v^i}{a(t)}\cdot\frac{\partial\ln\rho}{\partial x^i}\right)+\frac{\dot{a}(t)}{a(t)}\left(\frac{|v|^2}{c^2}+n-1\right) \nonumber\\
%	&+\frac{c^2-|v|^2}{a(t)}\left[\frac{\partial}{\partial t}\left(\frac{a(t)}{c^2-|v|^2}\right)+\sum_{i=1}^{n}\frac{\partial}{\partial x^i}\left(\frac{v^i}{c^2-|v|^2}\right)\right]=0.
%\end{align}

%Combining the Burgers equation, we get
\begin{align}
\frac{\partial\ln\rho}{\partial t}&+\sum_{i=1}^{n}\left(\frac{v^i}{a(t)}\cdot\frac{\partial\ln\rho}{\partial x^i}\right)+\frac{\dot{a}(t)}{a(t)}\left(\frac{u^2}{c^2}+n\right) \nonumber\\
&+\frac{1}{c^2-u}\frac{\partial u}{\partial t}+\sum_{i=1}^{n}\left(\frac{\partial v^i}{\partial x^i}\frac{1}{a(t)}+\frac{\partial u}{\partial x^i}\frac{v^i}{a(t)(c^2-u)}\right)=0.
\end{align}

Now we consider the Cauchy problem \eqref{ME2},
where $(\rho_{0}(x),v_0(x))\in C^{1}(\mathbb R^n)\times C^{2}(\mathbb R^n)$, and $\rho_0$ has a bounded 
	$C^1$  norm, while $v_0$ has a bounded $C^2$
	  norm. We further assume that there exists a positive constant $ Q_{0} $ such that
$$
\lVert \rho_0\rVert _{ C^{1}(\mathbb R^n)}+\lVert v_{0}\rVert _{ C^{2}(\mathbb R^n)}
\leqslant Q_{0}.
$$

Let $u=|v|^2$, define the characteristic curve $x^i=x^{i}(t,\alpha)$ as
\begin{equation}
\left\{\begin{aligned}
&\frac{\mathrm{d} x^i}{\mathrm{d}t}=\frac{v^i}{a(t)},\quad t>0, \\
&x^i(0,\alpha)=\alpha^i,\quad \alpha\in\mathbb{R}^n,
\end{aligned}\right.
\end{equation}
then along $x^{i}(t,\alpha)$, we have
\begin{equation}\label{3.3}
\left\{\begin{aligned}
&\frac{\mathrm{d}\ln\rho}{\mathrm{d}t}=-\frac{\dot{a}(t)}{a(t)}\left(\frac{u}{c^2}+n\right)-\frac{1}{c^2-u}\frac{\partial u}{\partial t}-\sum_{i=1}^{n}\left(\frac{\partial v^i}{\partial x^i}\frac{1}{a(t)}+\frac{\partial u}{\partial x^i}\frac{v^i}{a(t)(c^2-u)}\right),\quad t>0, \\
&\rho(0,\alpha)=\varepsilon\rho_0(\alpha),\quad \alpha\in\mathbb{R}^n.
\end{aligned}\right.
\end{equation}

Solving  \eqref{3.3} and combining with the second equation of system \eqref{ME2} we get easily
\begin{align}\label{3.4}
\rho(t,\alpha)=\varepsilon\rho_0(\alpha)\times\exp\left(-\int_{0}^{t}\left[\frac{\dot{a}(s)}{a(s)}\left(n-\sum_{i=1}^{n}\frac{(v^i)^2}{c^2}\right)+\frac{1}{a(s)}\sum_{i=1}^{n}\frac{\partial v^i}{\partial x^i}\right]\mathrm{d}s\right).
\end{align}
We can further get
\begin{align}\label{3.5}
\frac{\partial \rho(t,\alpha)}{\partial x^j}=\rho(t,\alpha)\int_{0}^{t}\left(\frac{\dot{a}(s)}{a(s)}\cdot\frac{\partial v^i}{\partial x^j}\cdot \frac{2v^i}{c^2}-\frac{1}{a(s)}\sum_{i=1}^{n}\frac{\partial^2 v^i}{\partial x^i\partial x^j}\right)\mathrm{d}s+\frac{\partial\rho_0}{\partial x^j}\frac{\rho}{\rho_0}.
\end{align}
In above, $\frac{\partial \rho_0(\alpha)}{\partial x^j}=\frac{\partial \rho_0(\alpha)}{\partial \alpha^l}\frac{\partial\alpha^l}{\partial x^j}$ can be discussed as the same as $\frac{\partial v^{i}(t,\alpha)}{\partial x^j}$ in Section \ref{sec:2}.
\begin{Remark}
From \eqref{3.4}, we can see that $\rho(t,x)>0$ in the existence domain of the classical solution, that means that $\ln\rho$ in \eqref{3.3} is well defined.
\end{Remark}
Based on above calculations, we are ready to prove Theorem \ref{thm:2}.
\begin{proof}
In the proof of Theorem \ref{thm:1}, when $a(t)$ satisfies Assumption $H_{1}$ or $H_{2}$, we know that $\det\left(\frac{\partial x^{i}(t,\alpha)}{\partial\alpha^j}\right)>0$ for all $t\in[0,+\infty)$. By Lemma \ref{rem:2.1}, the mapping $x=x(t,\alpha)$ is a global $C^{2}$
diffeomorphism. Consequently, the entries of the Jacobian matrix and Hessian matrix of $\alpha$ with respect to $x$, for fixed $t\in[0,+\infty)$, are all continuous functions.

According to the solution formulas of $\rho$ and $\partial \rho$, to demonstrate that $\rho$ and $\partial \rho$ are continuous on $[0,+\infty)\times \mathbb R^n$, we need to show that the integral functions of \eqref{3.4} and \eqref{3.5} are continuous on spacetime domain $[0,t]\times \mathbb R^n$ for all $0\leqslant t< +\infty$.

% It is easy to see that $\rho\in C([0,t]\times \mathbb R^n)$ for all $0\leq t<+\infty$, since the integral functions in \eqref{3.4} are all continuous ones according to Theorem \ref{thm:1}.

%Based on \eqref{3.5}, in order to show that  $\frac{\partial \rho}{\partial x^j}\in C([0,t]\times \mathbb R^n)$ for all $0\leq t<+\infty$, we also need to show the integral functions are continuous, which depends on the second order derivatives of $v$.

It is easy to see that the integral functions in \eqref{3.4} and \eqref{3.5} depend on the $v$, $\partial v$ and $\partial^2v$. By direct calculations, we have
\begin{equation}\label{v/x}
\frac{\partial v^i}{\partial x^j}=\sum_{l=1}^{n}\frac{\partial v^i}{\partial\alpha^l}\cdot\frac{\partial\alpha^l}{\partial x^j},
\end{equation}
and
\begin{equation}\label{second-order}
\frac{\partial^{2}v^i}{\partial x^j\partial x^k}=\sum_{l=1}^{n}\left[\sum_{s=1}^{n}\left(\frac{\partial^{2}v^i}{\partial \alpha^l\partial \alpha^s}\cdot\frac{\partial\alpha^s}{\partial x^k}\right)\cdot\frac{\partial\alpha^l}{\partial x^j}+\frac{\partial v^i}{\partial \alpha^l}\cdot\frac{\partial^{2}\alpha^l}{\partial x^j\partial x^k}\right],
\end{equation}
where
\begin{align}
\frac{\partial v^i}{\partial\alpha^l}(t,\alpha)&=\frac{c(\frac{\partial g_0^i(\alpha)}{\partial\alpha^l})a^{-1}(t)+cf_0(\alpha)[(\frac{\partial g_0^i(\alpha)}{\partial\alpha^l})f_0(\alpha)-g_0^i(\alpha)(\frac{\partial f_0(\alpha)}{\partial\alpha^l})]a^{-3}(t)}{[1+f_0^2(\alpha)a^{-2}(t)]^{3/2}}, \label{v for a} \\
\frac{\partial^{2}v^i}{\partial \alpha^l\partial \alpha^s}(t,\alpha)&=\frac{c(\frac{\partial^{2}g_0^i(\alpha)}{\partial \alpha^l\partial \alpha^s})a^{-1}(t)+c[2f_0(\alpha)(\frac{\partial f_0(\alpha)}{\partial\alpha^s})(\frac{\partial g_0^i(\alpha)}{\partial\alpha^l})+f_0^2(\alpha)(\frac{\partial^{2}g_0^i(\alpha)}{\partial \alpha^l\partial \alpha^s})]a^{-3}(t)}{[1+f_0^2(\alpha)a^{-2}(t)]^{3/2}} \nonumber\\
-&\frac{c[(\frac{\partial f_0(\alpha)}{\partial\alpha^s})(\frac{\partial f_0(\alpha)}{\partial\alpha^l}g_0^i(\alpha))+f_0(\alpha)(\frac{\partial^{2}f_0(\alpha)}{\partial \alpha^l\partial \alpha^s})g_0^i(\alpha)+f_0(\alpha)(\frac{\partial f_0(\alpha)}{\partial\alpha^l})(\frac{\partial g_0^i(\alpha)}{\partial\alpha^s})]a^{-3}(t)}{[1+f_0^2(\alpha)a^{-2}(t)]^{3/2}} \nonumber\\
-&\frac{3cf_0(\alpha)(\frac{\partial f_0(\alpha)}{\partial\alpha^s})\{(\frac{\partial g_0^i(\alpha)}{\partial\alpha^l})a^{-3}(t)+f_0(\alpha)[(\frac{\partial g_0^i(\alpha)}{\partial\alpha^l})f_0(\alpha)-g_0^i(\alpha)(\frac{\partial f_0(\alpha)}{\partial\alpha^l})]a^{-5}(t)\}}{[1+f_0^2(\alpha)a^{-2}(t)]^{5/2}}. \label{v/aa}
\end{align}
In above
\begin{align}
\frac{\partial^{2}f_0(\alpha)}{\partial \alpha^l\partial \alpha^s}&=\frac{c^2\varepsilon(\frac{\partial^{2}|v_0(\alpha)|}{\partial \alpha^l\partial \alpha^s})}{[c^2-\varepsilon^2|v_0(\alpha)|^2]^{3/2}}+\frac{3\varepsilon^3|v_0(\alpha)|^2(\frac{\partial|v_0(\alpha)|}{\partial\alpha^s})}{[c^2-\varepsilon^2|v_0(\alpha)|^2]^{5/2}}, \label{f_0two}\\
\frac{\partial^{2}g_0^i(\alpha)}{\partial \alpha^l\partial \alpha^s}&=\frac{c^2\varepsilon(\frac{\partial^{2}v_0^i(\alpha)}{\partial \alpha^l\partial \alpha^s})-\varepsilon^3[2|v_0(\alpha)|(\frac{\partial|v_0(\alpha)|}{\partial\alpha^s})(\frac{\partial v_0^i(\alpha)}{\partial\alpha^l})+|v_0(\alpha)|^2(\frac{\partial^{2}v_0^i(\alpha)}{\partial \alpha^l\partial \alpha^s})]}{[c^2-\varepsilon^2|v_0(\alpha)|^2]^{3/2}} \nonumber\\
+&\frac{\varepsilon^3[(\frac{\partial|v_0(\alpha)|}{\partial\alpha^s})(\frac{\partial|v_0(\alpha)|}{\partial\alpha^l})v_0^i(\alpha)+|v_0(\alpha)|(\frac{\partial^{2}v_0^i(\alpha)}{\partial \alpha^l\partial \alpha^s})v_0^i(\alpha)+|v_0(\alpha)|(\frac{\partial|v_0(\alpha)|}{\partial\alpha^l})(\frac{\partial v_0^i(\alpha)}{\partial\alpha^s})]}{[c^2-\varepsilon^2|v_0(\alpha)|^2]^{3/2}} \nonumber\\
+&\frac{3\varepsilon^2|v_0(\alpha)|(\frac{\partial |v_0|(\alpha)}{\partial\alpha^s})\{c^2\varepsilon(\frac{\partial v_0^i(\alpha)}{\partial\alpha^l})+\varepsilon^3|v_0(\alpha)|[v_0^i(\alpha)(\frac{\partial |v_0(\alpha)|}{\partial\alpha^l})-|v_0(\alpha)|(\frac{\partial v_0^i(\alpha)}{\partial\alpha^l})]\}}{[c^2-\varepsilon^2|v_0(\alpha)|^2]^{5/2}}. \label{g_0two}
\end{align}

When $a(t)$ satisfies Assumptions $H_1$ and $H_2$, the continuity and boundedness properties of $\partial v$ has been obtained in the proof of Theorem \ref{thm:1}. From \eqref{second-order}, the second order derivatives of $v$ rely on the initial data and the Jacobian and Hessian matrices of $\alpha$ with respect to $x$. These insights imply the integral functions in \eqref{3.4} and \eqref{3.5} are continuous on $[0,t]$ for all $0\leqslant t<+\infty$.

Thus, we can get that $(\rho(t,x),v(t,x))\in C^{1}([0,+\infty)\times \mathbb R^n)$.

In the following, we further show that $\rho(t,x)$ is not only a continuous function but also decays in relation to $a(t)$.

Since the initial data $v_0(x)\in C^2(\mathbb R^n)$ and are bounded, from \eqref{3.6} and \eqref{v/x}, it is easy to see that
\begin{equation}
|v^i|\leqslant C_1a^{-1}(t),
\end{equation}
\begin{equation}
\left|\frac{\partial v^i}{\partial \alpha^l}\right|\leqslant C_2a^{-1}(t)+C_3a^{-3}(t),
\end{equation}
% and
%\begin{equation}
%   \left|\frac{\partial^2 v^i}{\partial \alpha^l\partial\alpha^s}\right|\leq C_3(\alpha)a^{-1}(t)+C_4(\alpha)a^{-3}(t)+C_5(\alpha)a^{-5}(t),
%\end{equation}
where $C_i$ denote positive constants for $i=1,2,3$.

% Next, we can show that $\rho$ is bounded for all $t\in[0,+\infty)$, and when
%$t$ tends to infinity, we can further get that
%$\rho$ tends to $0$.

From \eqref{3.6} and \eqref{f_0estimate}-\eqref{g_0/alphaestimate}, we can find a positive constant $M$ such that
\begin{align}
\left|-\frac{\dot{a}(t)}{a(t)}\left(\sum_{i=1}^{n}\frac{(v^i)^2}{c^2}\right)+\frac{1}{a(t)}\sum_{i=1}^{n}\frac{\partial v^i}{\partial x^i}\right|&\leqslant \frac{M}{a^2(t)}\left(1+\frac{1}{a^2(t)}+\frac{\dot{a}(t)}{a(t)}\right)\leqslant \frac{2M}{a^2(t)}+\frac{M\dot{a}(t)}{a^3(t)},
\end{align}
in the second inequality, we have used  $a(t)\geqslant 1$.

Thus, By \eqref{H1inta} and \eqref{H2inta}, we can obtain that there exists $M_2$ such that for any $t\in [0,+\infty)$, we have
\begin{equation}
\int_{0}^t\left|-\frac{\dot{a}(t)}{a(t)}\left(\sum_{i=1}^{n}\frac{(v^i)^2}{c^2}\right)+\frac{1}{a(t)}\sum_{i=1}^{n}\frac{\partial v^i}{\partial x^i}\right|\mathrm{d}s\leqslant M_2+M_2\left(1-\frac{1}{2a^2(t)}\right)\leqslant 2M_2.
\end{equation}

For the remaining integral of \eqref{3.4}, it is easy to get
\begin{equation}
-\int_{0}^t\frac{\dot{a}(s)}{a(s)}n\mathrm{d}s=-n\ln a(t).
\end{equation}
Therefore, we can obtain
\begin{equation}
|\rho(t,\alpha)|\leqslant \varepsilon|\rho_0(\alpha)|e^{2M_2}a^{-n}(t),
\end{equation}
and
\begin{equation}
\lim_{t\to\infty}|\rho(t,\alpha)|\leqslant\lim_{t\to\infty}\varepsilon|\rho(\alpha)|e^{2M_2}a^{-n}(t)=0.
\end{equation}

Next, we estimate $\frac{\partial^2 \alpha^i}{\partial x^j\partial x^k}$, direct calculations give
\begin{align}
\frac{\partial^2 x^i}{\partial \alpha^j\partial \alpha^k}&=c\frac{\partial^2 g_0^i}{\partial\alpha^j\partial\alpha^k}\int_0^t\frac{a^{-2}(s)}{\sqrt{1+f_0^2(\alpha)a^{-2}(s)}}\mathrm{d}s-c\left(\frac{\partial g_0^i}{\partial\alpha^j}\cdot f_0\cdot\frac{\partial f_0}{\partial\alpha^k}\right. \nonumber\\
&\left.+\frac{\partial g_0^i}{\partial\alpha^k}\cdot f_0\cdot\frac{\partial f_0}{\partial\alpha^j}+g_0^i\cdot\frac{\partial f_0}{\partial\alpha^k}\cdot\frac{\partial f_0}{\partial\alpha^j}+g_0^i\cdot f_0\cdot\frac{\partial^2 f_0}{\partial\alpha^j\partial\alpha^k}\right)\int_0^t\frac{a^{-4}(s)}{(1+f_0^2(\alpha)a^{-2}(s))^{3/2}}\mathrm{d}s \nonumber\\
&+3g_0^i(f_0)^2\cdot\frac{\partial f_0}{\partial\alpha^j}\cdot\frac{\partial f_0}{\partial\alpha^k}\int_0^t\frac{a^{-6}(s)}{(1+f_0^2(\alpha)a^{-2}(s))^{5/2}}\mathrm{d}s,
\end{align}
by \eqref{f_0estimate}-\eqref{g_0/alphaestimate}, \eqref{H1inta}, \eqref{H2inta} and \eqref{f_0two}-\eqref{g_0two}, we can find a positive constant $M_3$, such that for any $t\in [0,+\infty)$, we have
\begin{equation}
\left|\frac{\partial^2 x^i}{\partial \alpha^j\partial \alpha^k}\right|\leqslant M_3,
\end{equation}
by Lemma \ref{2.4}, we obtain that there exists a positive constant $M_4$ such that.
\begin{equation}
\left|\frac{\partial^2 \alpha^i}{\partial x^j\partial x^k}\right|\leqslant M_4,
\end{equation}
Therefore, through \eqref{v/x}-\eqref{v/aa}, there exists a positive constant $M_5$ such that
\begin{equation}
\left|\frac{\dot{a}(s)}{a(s)}\cdot\frac{\partial v^i}{\partial x^j}\cdot \frac{v^i}{c^2}-\frac{1}{a(s)}\sum_{i=1}^{n}\frac{\partial^2 v^i}{\partial x^i\partial x^j}\right|\leqslant M_5\left(\frac{\dot{a}(t)}{a^3(t)}+\frac{\dot{a}(t)}{a^5(t)}+\frac{1}{a^2(t)}+\frac{1}{a^4(t)}+\frac{1}{a^6(t)}\right),
\end{equation}
which implies
\begin{equation}
\left|\int_{0}^{t}\left(\frac{\dot{a}(s)}{a(s)}\cdot\frac{\partial v^i}{\partial x^j}\cdot \frac{2v^i}{c^2}-\frac{1}{a(s)}\sum_{i=1}^{n}\frac{\partial^2 v^i}{\partial x^i\partial x^j}\right)\right|\leqslant M_6.
\end{equation}
where $M_6$ denotes a positive constant.

From the previous discussions and \eqref{3.5}, we can further get that 
%for some positive constant $M_7$
%\begin{equation}
%	\Big|\frac{\partial\rho}{\partial x^i}\Big|\leq M_7|\rho|
%\end{equation}
%and
\begin{equation}
\lim_{t\to+\infty}\frac{\partial\rho}{\partial x^i}(t,\alpha)=0.
\end{equation}

\end{proof}

\section{Proof of Theorem \ref{thm:3}}	\label{sec:4}
Combining Theorems \ref{thm:1} and \ref{thm:s}, we can conclude that when $a(t)$ satisfies Assumption $H_3$ or Assumption $H_4$ and $\displaystyle\inf_{\alpha}\det\left(\frac{\partial v_0}{\partial \alpha}\right)<0$, the velocity $v$ must blow up in finite time. While we do not know whether the density $\rho$ blows up or not since the solution formula is very complicated. Thus, we consider the spherically symmetric solution of system \eqref{ME1} in the following.

Let
\begin{equation*}
\rho(t,x)=\rho(t,r),\, v^{i}(t,x) = \textbf{v}(t,r)\frac{x^{i}}{r},
\end{equation*}
where
\begin{equation*}
\displaystyle r=\left(\sum_{i=1}^{n}{(x^i)^{2}}\right)^{\frac{1}{2}}.
\end{equation*}

The system \eqref{ME1} turns into
\begin{equation}
\left\{\begin{aligned}
&\frac{\partial}{\partial t}\left(\frac{a(t)\rho}{c^{2}-\textbf{v}^{2}}\right) + \frac{\partial}{\partial r}\left(\frac{\rho \textbf{v}}{c^{2}-\textbf{v}^{2}}\right)+ \left(\frac{n\dot{a}(t)\rho}{c^{2}-\textbf{v}^{2}} - \frac{\dot{a}(t)\rho}{c^{2}}\right) = 0,\\
&\frac{\partial}{\partial t}\left(\frac{a(t)\rho \textbf{v}}{c^{2}-\textbf{v}^{2}}\right) + \frac{\partial}{\partial r}\left(\frac{\rho \textbf{v}^2}{c^{2}-\textbf{v}^{2}}\right) + \left(\frac{n\dot{a}(t)\rho \textbf{v}}{c^{2}-\textbf{v}^{2}}\right) = 0,\\
&t=0: \rho(0,r)=\varepsilon \rho_0(r),\quad \textbf{v}(0,r)=\varepsilon \textbf{v}_0(r).
\end{aligned}\right.
\end{equation}

In the existence domain of the classical solution, the above system is equivalent to
\begin{equation}\label{5.2}
\begin{cases}
\frac{\partial\ln\rho}{\partial t}+\frac{\textbf{v}}{a(t)}\frac{\partial\ln\rho}{\partial r}+\frac{\dot{a}(t)}{a(t)}\left(\frac{\textbf{v}^2}{c^2}+n\right)+\frac{2\textbf{v}}{c^2-\textbf{v}^2}\frac{\partial\textbf{v}}{\partial t}+\frac{1}{a(t)}\left(\frac{2\textbf{v}^{2}}{c^2-\textbf{v}^2}+1\right)\frac{\partial\textbf{v}}{\partial r} = 0,\\
\frac{\partial \textbf{v}}{\partial t}+\frac{\textbf{v}}{a(t)}\frac{\partial \textbf{v}}{\partial r}+\frac{\dot{a}(t)}{a(t)}\textbf{v}-\frac{\dot{a}(t)}{a(t)}\frac{\textbf{v}^3}{c^2}=0,\\
t=0: \rho(0,r)=\varepsilon \rho_0(r),\quad \textbf{v}(0,r)=\varepsilon \textbf{v}_0(r).
\end{cases}
\end{equation}

It is easy to see the second equation is the same as the relativistic Burgers equation in dimension one, which has been discussed in the proof of Theorem \ref{thm:1}. Thus, we only need to investigate the first equation on density.

By the method of characteristics, we define the characteristic line $r=r(t,\alpha)$ as
\begin{equation}
\left\{\begin{aligned}
&\frac{\mathrm{d}r(t,\alpha)}{\mathrm{d}t}=\frac{\textbf{v}(t,r(t,\alpha))}{a(t)},\quad t>0, \\
&r(0,\alpha)=\alpha.
\end{aligned} \right.
\end{equation}

Then along the characteristic line $r(t,\alpha)$, we have
\begin{equation}
\left\{\begin{aligned}	&\frac{\mathrm{d}\ln\rho}{\mathrm{d}t}=-\left[\frac{\dot{a}(t)}{a(t)}\left(\frac{\textbf{v}^2}{c^2}+n\right)+\frac{2\textbf{v}}{c^2-\textbf{v}^2}\frac{\partial\textbf{v}}{\partial t}+\frac{1}{a(t)}\left(\frac{2\textbf{v}^{2}}{c^2-\textbf{v}^2}+1\right)\frac{\partial\textbf{v}}{\partial r}\right],\\
&\rho(0,\alpha)=\varepsilon\rho_0(\alpha).
\end{aligned}\right. \label{3.14}
\end{equation}

Solving  \eqref{3.14}, we have
\begin{equation}\label{5.19}
\rho(t,\alpha)=\varepsilon\rho_0(\alpha)\exp\left\{-\int_{0}^{t}\left[\frac{\dot{a}(s)}{a(s)}\left(n-\frac{\textbf{v}^2}{c^2}\right)+\frac{1}{a(s)}\frac{\partial\textbf{v}}{\partial r}\right]\mathrm{d}s\right\}.
\end{equation}

We can further get
\begin{equation}\label{5.6}
\frac{\partial \rho}{\partial r}(t,\alpha) =  \rho(t,\alpha) \Big\{ \int_{0}^{t}\left[\frac{\dot{a}(s)}{a(s)} \frac{2\textbf{v}}{c^2}\frac{\partial\textbf{v}}{\partial r}
- \frac{1}{a(s)}\frac{\partial^2\textbf{v}}{\partial r^2} \right]\, \mathrm{d}s
+ \frac{\partial \rho_0}{\partial r} \frac{1}{\rho_0}\Big\} .
\end{equation}

Next we calculate $ \frac{\partial^2 r}{\partial \alpha^2} $ and $ \frac{\partial^2 \textbf{v}}{\partial \alpha^2} $
\begin{align*}
\frac{\partial r}{\partial \alpha}(t,\alpha) = &1+c{f_{0}}'(\alpha)\int_{0}^{t}\frac{a^{-2}(s)}{[1+{f_{0}^{2}}(\alpha)a^{-2}(s)]^{3/2}}\mathrm{d}s,\\
\frac{\partial^2 r}{\partial \alpha^2}(t,\alpha)=&c\int_{0}^{t}\frac{[{f_{0}}''(\alpha)-3{f_{0}}(\alpha)({f_{0}}'(\alpha))^2]a^{-2}(s)+{f_{0}}''(\alpha){f_{0}}(\alpha)a^{-4}(s)}{[1+{f_{0}^{2}}(\alpha)a^{-2}(s)]^{5/2}}\mathrm{d}s,\\
\frac{\partial \textbf{v}}{\partial \alpha}(t,\alpha)=&\frac{c{f_{0}}'(\alpha)a^{-1}(t)}{[1+{f_{0}^{2}}(\alpha)a^{-2}(t)]^{3/2}},\\
\frac{\partial^2 \textbf{v}}{\partial \alpha^2}(t,\alpha)=&\frac{c{f_{0}}''(\alpha)a^{-1}(t)+c[{f_{0}}''(\alpha){f_{0}^{2}}(\alpha)-3{f_{0}}(\alpha)({f_{0}}'(\alpha)^2)]a^{-3}(t)}{[1+{f_{0}^{2}}(\alpha)a^{-2}(t)]^{5/2}},
\end{align*}
where $f_0(\alpha)=\frac{\varepsilon |\textbf{v}_0(\alpha)|}{\sqrt{c^2-\varepsilon^2|\textbf{v}_0(\alpha)|^2}}$ is first defined in \eqref{|v|^2}.

Since $ \textbf{v}_{0} $ has bounded $ C^{2} $ norm, we have
\begin{equation*}
\Vert f_{0}''\Vert_{C^0(\mathbb{R^+})}=\left\Vert\frac{c^2\varepsilon \textbf{v}_0''(c^2-\varepsilon^2\textbf{v}_0^2)+3c^2\varepsilon^3\textbf{v}_0(\textbf{v}_0')^2}{(c^2-\varepsilon^2\textbf{v}_0^2)^{5/2}}\right\Vert_{C^{0}(\mathbb R^+)}\leqslant\frac{c^4\varepsilon  Q_0+3c^2\varepsilon^3Q_0^3}{(c^2-\varepsilon^2Q_0^2)^{5/2}}.
\end{equation*}

So there exists $Q_2>0$, such that
\begin{align*}
\Vert {f_{0}}\Vert_{C^2(\mathbb{R^+})}\leqslant Q_2,
\end{align*}
and then
\begin{align*}
&\Big\Vert \frac{\partial^2 r}{\partial \alpha^2} (t,\cdot)\Big\Vert_{C^0(\mathbb{R^{+}})}\leqslant c(Q_2+3Q_2^3)\int_{0}^{t}a^{-2}(s)\mathrm{d}s+cQ_2^3\int_{0}^{t}a^{-4}(s)\mathrm{d}s,\\
&\Big\Vert \frac{\partial \textbf{v}}{\partial \alpha} (t,\cdot)\Big\Vert_{C^0(\mathbb{R^+})}\leqslant \frac{cQ_2}{a(t)}\leqslant cQ_2,\\
&\Big\Vert \frac{\partial^2 \textbf{v}}{\partial \alpha^2} (t,\cdot)\Big\Vert_{C^0(\mathbb{R^{+}})}\leqslant\frac{cQ_2}{a(t)}+\frac{4cQ_2^3}{a^3(t)}\leqslant 5cQ_2.
\end{align*}
It is easy to see
\begin{align*}
\frac{\partial \textbf{v}}{\partial r}=\frac{\partial \textbf{v}}{\partial \alpha}\cdot
\frac{\partial \alpha}{\partial r}
\end{align*}
and
\begin{align}\label{vrr}
\frac{\partial^2 \textbf{v}}{\partial r^2}=\left(\frac{\partial^2 \textbf{v}}{\partial \alpha^2}\frac{\partial r}{\partial \alpha}-\frac{\partial \textbf{v}}{\partial \alpha}\frac{\partial^2 r}{\partial \alpha^2}\right)\cdot\left(\frac{\partial r}{\partial \alpha}\right)^{-3}.
\end{align}

To analyze the global existence and potential blowup of the classical solution, we divide the discussion into two cases based on the monotonicity of $\textbf{v}_0(\alpha)$.

\textbf{Case I.}\quad if $\displaystyle\inf_{\alpha}\, \textbf{v}_0'(\alpha)\geqslant0$.

When $\textbf{v}_0'(\alpha)=0$ for some $\alpha$, then along such characteristic line, we have
\begin{align*}
\frac{\partial r}{\partial \alpha}(t,\alpha)=1, \\
\frac{\partial \textbf{v}}{\partial \alpha}(t,\alpha)=0.
\end{align*}
We can easily get $\frac{\partial \textbf{v}}{\partial r}(t,\alpha)=0$ and $\frac{\partial^2 \textbf{v}}{\partial r^2}(t,\alpha)$ are bounded since $\frac{\partial^2 \textbf{v}}{\partial \alpha^2}(t,\alpha)$ is bounded.

When $\textbf{v}_0'(\alpha)>0$, we have
$$ 	\frac{\partial \textbf{v}}{\partial r}=\frac{\partial \textbf{v}}{\partial \alpha}\cdot
\frac{\partial \alpha}{\partial r}>0 $$
and

$$\displaystyle\lim_{t\to+\infty}	\frac{\partial r}{\partial \alpha}(t,\alpha)=+\infty,$$
we then get  $\frac{\partial \textbf{v}}{\partial r},\frac{\partial^2 \textbf{v}}{\partial r^2}$ are bounded for all $t\in[0,+\infty)$ in this case since
\begin{align*}
\lim_{t\to+\infty}\frac{\partial \textbf{v}}{\partial r}&=\lim_{t\to+\infty}\left(\frac{\partial \textbf{v}}{\partial \alpha}\cdot
\frac{\partial \alpha}{\partial r}\right)=0, \\
\lim_{t\to+\infty}\frac{\partial ^2\textbf{v}}{\partial r^2}&=\lim_{t\to+\infty}\Big[\left(\frac{\partial^2 \textbf{v}}{\partial \alpha^2}\cdot\frac{\partial r}{\partial \alpha}-\frac{\partial \textbf{v}}{\partial \alpha}\frac{\partial^2 r}{\partial \alpha^2}\right)\left(\frac{\partial r}{\partial \alpha}\right)^{-3}\Big]\\
&=\lim_{t\to+\infty}\Big[\frac{\partial^2 \textbf{v}}{\partial \alpha^2}\left(\frac{\partial r}{\partial \alpha}\right)^{-2}\Big]-\frac{1}{3}\lim_{t\to+\infty}\frac{\partial \textbf{v}}{\partial \alpha}\lim_{t\to+\infty}\frac{\frac{d}{dt}\left(\frac{\partial^2 r}{\partial \alpha^2}\right)}{(\frac{\partial r}{\partial \alpha})^2\frac{d}{dt}\left(\frac{\partial r}{\partial \alpha}\right)}=0,
\end{align*}
here we have used $	\frac{\partial \textbf{v}}{\partial \alpha}(t,\alpha),	\frac{\partial ^2\textbf{v}}{\partial \alpha^2}(t,\alpha)$ are bounded and the L'Hopital's rule.

Since the initial data $\textbf{v}_0(x)\in C^2(\mathbb R^n)$ and are bounded, combining \eqref{u}, we have
\begin{equation*}
|\textbf{v}(t,\alpha)|^2\leqslant M\varepsilon^2,
\end{equation*}
where M is a positive constant.

Set $ M_8(t,\alpha)= n-\frac{\textbf{v}^2}{c^2}$, there exists a positive constant $0<C_6<n$ such that
\begin{equation*}
|M_8(t,\alpha)|\geqslant n-\frac{M\varepsilon^2}{c^2}\geqslant C_6,
\end{equation*}
provided $\varepsilon$ is suitably small.
%\begin{equation}
%\textbf{v}^2\leqslant C_4a^{-2}(t),
%\end{equation}
%\begin{equation}
%\left|\frac{\partial %\textbf{v}}{\partial %r}\right|\leqslant C_5a^{-1}(t),
%\end{equation}
%Then we can find a positive constant $M$ such that
%\begin{align}
%	\left|-\frac{\dot{a}(t)}{a(t)}\frac{\textbf{v}^2}{c^2}+\frac{1}{a(t)}\frac{\partial \textbf{v}}{\partial r}\right|&\leqslant \frac{M}{a^2(t)}+M\frac{\dot{a}(t)}{a^3(t)},
%\end{align}
%so
%\begin{align}
%\int_{0}^t\left|-\frac{\dot{a}(s)}{a(s)}\frac{\textbf{v}^2}{c^2}+\frac{1}{a(s)}\frac{\partial \textbf{v}}{\partial r}\right|\mathrm{d}s&\leqslant \int_{0}^t\frac{M}{a^2(s)}\mathrm{d}s+\int_{0}^t\frac{M\dot{a}(s)}{a^3(s)}\mathrm{d}s\\&\leqslant %\int_{0}^t\frac{M}{a^2(s)}\mathrm{d}s
%\end{align}
%For the remaining integral of \eqref{5.6}, it is easy to get
%begin{equation}
%	-\int_{0}^t\frac{\dot{a}(s)}{a(s)}n\mathrm{d}s=-n\ln a(t).
%\end{equation}

%Then we have
%\begin{equation}
%|\rho(t,\alpha)|\leqslant \varepsilon|\rho_0(\alpha)|e^{2M_2}a^{-n}(t),
%\end{equation}
% and
% \begin{equation}
%	\lim_{t\to\infty}|\rho(t,\alpha)|\leqslant\lim_{t\to\infty}\varepsilon|\rho(\alpha)|e^{2M_2}a^{-n}(t)=0.
% \end{equation}
Then we have
\begin{align}
|\rho(t,\alpha)|&=\varepsilon|\rho_0(\alpha)|e^{-\int_{0}^{t}\frac{\dot{a}(s)}{a(s)}M_8\mathrm{d}s}e^{-\int_{0}^{t}\frac{1}{a(s)}\frac{\partial\textbf{v}}{\partial r}\mathrm{d}s}\\ \nonumber
&\leqslant\varepsilon\rho_0(\alpha)a^{-C_6}(t),\nonumber
\end{align}
here we have used $\frac{ \partial \textbf{v}}{\partial r}\geqslant0 $.

Obviously, when $a(t)$ satisfies Assumption $H_3$, we have
\begin{equation}\label{uu}
\lim_{t\to+\infty}|\rho(t,\alpha)|\leqslant\lim_{t\to+\infty}\varepsilon|\rho_0(\alpha)|a^{-C_6}(t)=0,
\end{equation}
while when $a(t)$ satisfies Assumption $H_4$, we only get
\begin{equation}\label{uu2}
|\rho(t,\alpha)|\leqslant\varepsilon|\rho_0(\alpha)|.
\end{equation}

Combining above discussions, the integral functions in \eqref{5.6} and \eqref{5.19} are continuous and bounded for all $t\in[0,+\infty)$. Thus,
we finally get $ \rho(t,r)\in C^{1}([0,+\infty)\times \mathbb R^{+}) $ and has bounded $C^0$ norm if $\displaystyle\inf_{\alpha}\, \textbf{v}_0'(\alpha)\geqslant0$.

\textbf{Case II.} \quad if $\displaystyle\inf_{\alpha}\, \textbf{v}_0'(\alpha)<0$.

If there exists an $\alpha$ such that $\textbf{v}_0'(\alpha)<0$, we have ${f_{0}}'(\alpha)<0$.

Let $t_2$ be the blow up time of $\textbf{v}$ investigated in Section \ref{sec:2}, then we get
\begin{equation*}
\displaystyle\lim_{t\to t_2^-}\frac{\partial r}{\partial \alpha}(t,\alpha)=0^+,
\end{equation*}
and
\begin{equation*}
\lim_{t\to t_2^-}\frac{\partial \textbf{v}}{\partial r}(t,\alpha)=\lim_{t\to t_2^-}\frac{\frac{\partial \textbf{v}}{\partial \alpha}(t,\alpha)}{\frac{\partial r}{\partial \alpha}(t,\alpha)}=\lim_{t\to t_2^-}\frac{c{f_{0}}'(\alpha)a^{-1}(t)}{[1+{f_{0}^{2}}(\alpha)a^{-2}(t)]^{3/2}}\cdot\frac{1}{\frac{\partial r}{\partial \alpha}(t,\alpha)}=-\infty.
\end{equation*}
By the L'Hopital's rule, we can also get that 
\begin{equation*}
\lim_{t\to t_2^-}\frac{\frac{\partial \textbf{v}(t,\alpha)}{\partial r}}{(t_2-t)^{-1}}=\lim_{t\to t_2^-}\frac{c{f_{0}}'(\alpha)a^{-1}(t)}{[1+{f_{0}^{2}}(\alpha)a^{-2}(t)]^{3/2}}\lim_{t\to t_2^-}\frac{-1}{\frac{cf_0'(\alpha)a^{-2}(t)}{[1+{f_{0}^{2}}(\alpha)a^{-2}(t)]^{3/2}}}=-a(t_2)<0,
\end{equation*}
which means that the blowup rate of $\nabla v$ is $O(\frac{1}{t_2-t})$.

By \eqref{5.19}, we can get
\begin{equation*}
\lim_{t\to t_2^-}\rho(t,\alpha)=+\infty,
\end{equation*}
since $\int_{0}^{t_2}(t_2-t)^{-1}\mathrm{d}t=+\infty$. Thus, in this case, $\rho$ blows up, when $ t$ tends to $t_2$.

At the end of this section, we have several remarks on the spherically symmetric solution.
\begin{Remark}
When $a(t)$ satisfies Assumption $H_4$,
\begin{equation*}	\label{5.20}
\rho(t,\alpha)=\varepsilon\rho_0(\alpha)e^{-\int_{0}^{t}\frac{\partial\textbf{v}}{\partial r}\mathrm{d}s}.
\end{equation*}

If $\displaystyle\inf_{\alpha}\, \textbf{v}_0'(\alpha)\geqslant0$ and when $ t $ tends to $+\infty$, we cannot obtain the decay estimate similar to \eqref{uu}. This can also be checked by the following example.

If we let $ \rho_0(r) = 1 $, then
\begin{equation*}
\left\{
\begin{aligned}
	&\rho(t,r)\equiv \varepsilon,\\
	&\textbf{v}(t,r)\equiv 0,
\end{aligned}\right.
\end{equation*}
is a solution to system \eqref{5.2}. However, $\rho$ satisfies
\begin{equation*}
\lim_{t\to+\infty}\rho(t,\alpha)=\varepsilon\neq 0.
\end{equation*}
\end{Remark}

\begin{Remark}
By \eqref{5.6}, we have
\begin{align*}
\left(\frac{\partial^2 \textbf{v}}{\partial \alpha^2} \frac{\partial r}{\partial \alpha}\right. - &\left.\frac{\partial \textbf{v}}{\partial \alpha} \frac{\partial^2 r}{\partial \alpha^2} \right) = \left( \frac{c{f_{0}}''(\alpha)a^{-1}(t) + c[{f_{0}}''(\alpha){f_{0}^{2}}(\alpha) - 3{f_{0}}(\alpha)({f_{0}}'(\alpha))^2]a^{-3}(t)}{[1+{f_{0}^{2}}(\alpha)a^{-2}(t)]^{5/2}} \right) \\
&\times \left( 1 + c{f_{0}}'(\alpha) \int_{0}^{t} \frac{a^{-2}(s)}{[1+{f_{0}^{2}}(\alpha)a^{-2}(s)]^{3/2}} \mathrm{d}s \right)- \left( \frac{c{f_{0}}'(\alpha)a^{-1}(t)}{[1+{f_{0}^{2}}(\alpha)a^{-2}(t)]^{3/2}} \right) \\
&\times \left( c \int_{0}^{t} \frac{[{f_{0}}''(\alpha)-3{f_{0}}(\alpha)({f_{0}}'(\alpha))^2]a^{-2}(s)+{f_{0}}''(\alpha){f_{0}}(\alpha)a^{-4}(s)}{[1+{f_{0}^{2}}(\alpha)a^{-2}(s)]^{5/2}} \mathrm{d}s \right),
\end{align*}
because the orders of the two terms on the right-hand side are as the same as $\frac{1}{a(t)}\int_{0}^{t}\frac{1}{a^2(s)}\mathrm{d}s$, however, the sign of their coefficients cannot be determined. Therefore, we cannot conclude whether
$\frac{\partial \rho}{\partial r}$ is bounded or not when $t$ tends to $+\infty$.
\end{Remark}

\begin{Remark} From the spherically symmetric solution, if $\displaystyle\inf_{\alpha}\, \textbf{v}_0'(\alpha)<0$, we find that the density will concentrate in finite time, namely $\rho$ itself blows up in finite time, while $ \textbf{v} $ is bounded and its derivative blows up.
\end{Remark}

\noindent{\Large {\bf Acknowledgements.}} The authors would like to thank the anonymous reviewer for providing valuable comments and suggestions. This work is partially supported by the Outstanding Youth Fund of Zhejiang Province (Grant No. LR22A010004), the NSFC (Grant No. 12071435), and the Fundamental Research Fund of Zhejiang Sci-Tech University.
\\
\noindent{\Large {\bf Data availability statement.}} No new data were created or analysed in this study.\\

\noindent{\Large {\bf Conflict of interest}} The authors declare that they have no conflict of interest to this work.\\

\noindent{\Large {\bf Ethical approval}} Not applicable.\\

\noindent{\Large {\bf Consent to participate}} Not applicable

\end{CJK}


\begin{thebibliography}{999}
%\bibitem{Alinhac} Alinhac, S.: {\it The null condition for quasilinear wave equations in two space dimensions I}, Invent. Math., 145(2001) 597-618.
%\bibitem{Rendall} U. Brauer,  A. Rendall, O. Reula, {\it The cosmic no-hair theorem and the non-linear stability of homogeneous Newtonian cosmological models}, Class. Quantum Grav., \textbf{11} (1994), 2283-2296.
\bibitem{Todd3} F. Beyer, M. Elliot, T. Oliynyk, {\it Future instability of FLRW fluid solutions for linear equations of state $ p= K\rho $ with $ 1/3< K< 1 $}, Physical Review D., \textbf{107(10)}, (2023), Art. 104030, 14 pp.
\bibitem{Courant} R. Courant, K. Fridriches, {\it Supersonic flow and shock waves}, New York, 1948.

\bibitem{Christodoulou} D. Christodoulou, {\it The formation of shocks in 3-dimensional fluids}, EMS Monographs in Mathematics, Z$\ddot{u}$rich, 2007.
\bibitem{Christodoulou-Lisibach} D. Christodoulou, A. Lisibach, {\it Shock development in spherical symmetry}, Ann. PDE., \textbf{2}, (2016), Art. 3, 246 pp.
\bibitem{Christodoulou-Miao} D. Christodoulou, S. Miao, {\it Compressible flow and Euler's equations}, Surveys of Modern Mathematics, 9. International Press, Somerville, MA; Higher Education Press, Beijing, 2014.

%\bibitem{Ding-Witt-Yin} B. Ding, I. Witt, H. Yin, {\it The global smooth symmetric solution to 2-D full compressible Euler system of Chaplygin gases}, J. Differ. Eqs., \textbf{258} (2015) 445-482.
\bibitem{Fajman} D. Fajman, M. Ofner and Z. Wyatt, {\it Slowly expanding stable dust spacetimes}, Archive for Rational Mechanics and Analysis, \textbf{248(5)}, (2024), 83.

\bibitem{F-T} D. Fajman, M. Ofner, T. Oliynyk, Z. Wyatt, {\it The Stability of Relativistic Fluids in Linearly Expanding Cosmologies}, Int. Math. Res. Not. IMRN, \textbf{2024(5)}, (2024), 4328–4383.
\bibitem{F-T1} D. Fajman, T. Oliynyk, Z. Wyatt, {\it Stabilizing relativistic fluids on spacetimes with non-accelerated expansion}, Comm. Math. Phys., \textbf{383}, (2021), 401-426.
\bibitem{FOOW24} D. Fajman, M. Maliborski, M. Ofner, T. Oliynyk, Z. Wyatt, {\it Phase transition between shock formation and stability in cosmological fluids}, arXiv:2405.03431.

\bibitem{FOOW25} D. Fajman, M. Ofner, T. Oliynyk, Z. Wyatt, {\it Stability of fluids in spacetimes with decelerated expansion}, arXiv:2501.12798.

\bibitem{Friedrich} H. Friedrich, {
\it Sharp asymptotics for Einstein-$\lambda$-Dust flows}, Comm. Math. Phys., \textbf{350(2)}, (2017), 803-844.
%\bibitem{Godin} P. Godin, {\it Global existence of a class of smooth 3D spherically symmetric flows of Chaplygin gases with variable entropy}, J. Math. Pures Appl., \textbf{87} (2007) 91-117.

%\bibitem{Guo} Y. Guo, S. Tahvildar-Zedeh, {\it Formation of singularities in relativistic fluid dynamics and spherically symmetric plasma dynamics}, Contemp. Math., \textbf{238} (1999) 151-161.
\bibitem{Grassin-Serre} M. Grassin, D. Serre, {\it Existence de solutions globales et r\'{e}guli\`{e}res aux \'{e}quations d'Euler pour un gaz parfait isentropique}, Comptes Rendus de l'Acad\'{e}mie des Sciences-Series I-Mathematics, \textbf{325(7)}, (1997), 721-726.

\bibitem{H-J} M. Had\v{z}i\'{c}, J. Speck, {\it The global future stability of the FLRW solutions to the dust-Einstein system with a positive cosmological constant}, J. Hyperbolic Differ. Equ., \textbf{12(01)}, (2015), 87-188.


\bibitem{H-K-S-W} G. Holzegel, S. Klainerman, J. Speck, Y. Wong, {\it Small-data shock formation in solutions to 3D quasilinear wave equations: an overview}, J. Hyperbolic Differ. Equ., \textbf{13}, (2016), 1-105.
\bibitem{Huo-Wei} S. Huo, C. Wei,
{\it Classical solutions to relativistic Burgers equations in FLRW space-times}, Sci. China Math., \textbf{63}, (2020), 357-370.
\bibitem{Huo-Wei1} S. Huo, C. Wei, {\it Classical solution to relativistic Burger's equation in SdS and SAdS space-times}, J. Math. Phys., \textbf{60(2)}, (2019), 021504.
\bibitem{Kong} D. Kong, {\it Cauchy problem for quasilinear hyperbolic system}, MSJ Mem., 2000.
%\bibitem {John} John, F.:{\it Nonlinear wave equations, formation of singularities}, Pitcher Lectures delivered at Lehigh University, 1989.
%\bibitem{Kong} Kong, D.X., Tsuji, M.: {\it Global solutions for $2\times2$ hyperbolic systems with linearly degenerate characteristics}, Funkcial. Ekvac., 42(1999) 129-155.
%\bibitem{Kong-Liu-Wang}  D.X. Kong, K. Liu, Y. Wang, {\it Global existence of smooth solutions to two-dimensional compressible isentropic Euler equations for Chaplygin gases}, Sci. China Math., \textbf{53} (2010) 719-738.
%\bibitem{Lax} Lax, P.: {\it Development of singularity of solutions of nonlinear hyperbolic partial differential equations}, J. Math. Phys., 5(1964) 611-613.
%\bibitem{Lax1} Lax, P.: {\it  Hyperbolic systems of conservation laws in several space variables}, Current topics in partial differential equations, (1986) 327-341.
%\bibitem{Liu} Liu, T.: {\it The development of singularity in the nonlinear waves for quasi-linear hyperbolic partial differential equations}, J. Differ. Eqs., 33(1979) 92-111.
%\bibitem{Lei}  Z. Lei, Y. Du, Q. Zhang, {\it Singularities of solutions to compressible Euler equations with vacuum}, Math. Res. Lett. \textbf{20} (2013) 41-50.
%\bibitem{Lei-Lin-Zhou} Lei, Z., Lin, F., Zhou, Y.: {\it Global solutions of the evolutionary Faddeev model with small initial data},  Acta Math. Sin. (Engl. Ser.) 27(2011) 309-328.
%\bibitem{Lei-Wei}  Z. Lei, C. H. Wei, {\it Global radial solutions to 3D relativistic Euler equations for non-isentropic Chaplygin gases}. Math. Ann. \textbf{367} (2017), 1363-1401.
\bibitem{LeFloch} P. G. LeFloch, H.  Makhlof, B. Okutmustur, {\it Relativistic Burgers equations on curved space-times: Derivation and finite
volume approximation}, SIAM J Numer. Anal., \textbf{50(4)}, (2012), 2136-2158.

\bibitem{LeFloch-Wei} P. G. LeFloch, C. H. Wei, {\it Nonlinear stability of self-gravitating irrotational Chaplygin fluids in a FLRW geometry}. A. I. H.P. Non Lineair. \textbf{38(3)}, (2021), 787-814.
\bibitem{Li} T. Li {\it Global classical solutions for quasilinear hyperbolic systems}, Research in applied Mathematics, 1994.

\bibitem{Liu-Wei} C. Liu, C. Wei, {\it  Future stability of the FLRW spacetime for a large class of perfect fluids}, Ann. Henri Poincare, \textbf{22(3)}, (2021), 715-770.
\bibitem{Todd1} E. Marshall, T. Oliynyk, {\it  On the stability of relativistic perfect fluids with linear equations of state  $p=Kρ$ where $1/3<K<1$}, Lett. Math. Phys., \textbf{113(5)}, (2023), 104030.
\bibitem{Majda} A. Majda, {\it Compressible fluid flow and systems of conservation laws in several space variables}, Comm Pure Appl. Math., \textbf{28}, (1975), 607-676.
\bibitem{Miao-Yu} S. Miao, P. Yu, {\it On the formation of shocks for quasilinear wave equations}, Invent. Math., \textbf{207}, (2017), 697-831.
%\bibitem{Todd1} Todd A. Oliynyk, {\it Future stability of the FLRW fluid solutions in the presence of a positive cosmological constant}, Comm. Math. Phys. \textbf{346} (2016) 293-312.
%\bibitem{Todd} Todd A. Oliynyk, {\it The Newtonian limit on cosmological scales}. Comm. Math. Phys. \textbf{339} (2015) 455-512.

%\bibitem{Pan} R. Pan, J. Smoller, {\it Blowup of smooth solutions for relativistic Euler equations}, Commun. Math. Phys., \textbf{262} (2006) 729-755.
%\bibitem{PB}  B. Perthame, {\it Non-existence of global solutions to Euler-Possion equations for repulsive forces}, Japan J. Appl. Math., \textbf{7} (1990), 363-367.
%\bibitem{Ringstrom} H. Ringstr\"{o}m, {\it Power law inflation}, Comm. Math. Phys., \textbf{290} (2009), 155-218.

%\bibitem{Smoller-Temple} Smoller, J., Temple, B.: {\it Global solutions of the relativistic Euler equations}, Commun. Math. Phys., 156(1993) 67-99.
%\bibitem{Smoller} Smoller, J.: {\it Shock waves and reaction diffusion equations}, Berlin-Heidelberg-New-York: Springer-Verlag, 2nd Edition, 1993.
%\bibitem{Sideris} T. Sideris, {\it Formation of singularity in three-dimensional compressible fluids}, Commun. Math. Phys., \textbf{101} (1985) 475-485.
%\bibitem{Sideris2} T. Sideris, {\it Formation of singularities of solutions to nonlinear hyperbolic equations}, Arch. Ration. Mech. Anal., \textbf{86} (1984) 369-381.
%\bibitem{Sideris4} Sideris, T.: {\it Nonresonance and global existence of prestressed nonlinear elastic waves}, Ann. of Math., 151(2000) 849-874.
%    \bibitem{Sideris1} Sideris, T.: {\it Delayed singularity formation in 2D compressible flow}, Amer. J. Math., 119(1997) 371-422.
%\bibitem{Sideris6} T. Sideris, {\it Spreading of the free boundary of an ideal fluid in a vacuum}, J. Differ Eqs., \textbf{257} (2014) 1-14.
\bibitem{Ringstrom} H. Ringstr\"{o}m, {\it Future stability of the Einstein-non-linear scalar field system}, Invent. Math., \textbf{173(1)}, (2008), 123-208.

\bibitem{Ringstrom1} H. Ringstr\"{o}m, {\it  Power law inflation}, Comm. Math. Phys., \textbf{290(1)}, (2009), 155-218.

\bibitem{Rodnianski} I. Rodnianski, J. Speck,   {\it The nonlinear future stability of the FLRW family of solutions to the irrotational Euler–Einstein system with a positive cosmological constant}, Journal of the European Mathematical Society., \textbf{15(6)}, (2013), 2369-2462.

\bibitem{Serre} D. Serre, {\it Solutions classiques globales des \'{e}quations d'Euler pour un fluide parfait compressible}, Annales de l'institut Fourier, \textbf{47(1)}, (1997) 139-153.
\bibitem{Sideris-Wang} T. Sideris, B. Thomases, D. Wang, {\it Long time behavior of solutions to the 3D compressible Euler equations with damping}, Comm. Partial Differential Equations, \textbf{28}, (2003), 795-816.
\bibitem{Speck} J. Speck, {\it The nonlinear future stability of the FLRW family of solutions to the Euler-Einstein system with a positive cosmological constant}, Selecta Math., \textbf{18}, (2012), 633-715.
%\bibitem{AS} A. S. Tahvildar-Zadeh, {\it Relativistic and nonrelativistic elastodynamics with small shearstrains}, Ann. Inst. H. Poincar$\acute{e}$ Phys. Th$\acute{e}$or. \textbf{69} (1998) 275-307.
%\bibitem{Wei-Han} C. H. Wei, B. Han, {\it Spreading of the free boundary of relativistic Euler equations in a vacuum}, Math. Res. Lett. Revised.
%\bibitem{Zhou} Y. Zhou, W. Han, {\it Blow up for some semilinear wave equations in multi-space dimensions}, Comm. Partial Differential Equations, \textbf{39} (2014) 652-665.
%\bibitem{Zhou1} Y. Zhou, W. Han, {\it Life-span of solutions to critical semilinear wave equations}, Comm. Partial Differential Equations, \textbf{39} (2014) 439-451.

\bibitem{Speck} J. Speck, {\it The stabilizing effect of spacetime expansion on relativistic fluids with sharp results for the radiation equation of state}, Archive for Rational Mechanics and Analysis., \textbf{210}, (2013), 535-579.

\bibitem{Speck1} J. Speck, {\it The nonlinear future stability of the FLRW family of solutions to the Euler–Einstein system with a positive cosmological constant}, Selecta Mathematica., \textbf{18(3)}, (2012), 633-715.


\bibitem{Todd2} T. Oliynyk, {\it Future Global Stability for Relativistic Perfect Fluids with Linear Equations of State $p=K\rho$ where $ 1/3<K<1/2 $}, SIAM Journal on Mathematical Analysis., \textbf{53(4)} (2021) 4118-4141.
\bibitem{Oliynyk1} T. Oliynyk, {\it Future stability of the FLRW fluid solutions in the presence of a positive cosmological constant}, Communications in Mathematical Physics., \textbf{346}, (2016), 293-312.

\bibitem{Wei} C. Wei, {\it  Stabilizing effect of the power law inflation on isentropic relativistic fluids}, J. Differential Equations, \textbf{265(8)}, (2018), 3441-3463.
\bibitem{Wei1} C. Wei, {\it Classical solutions to the relativistic Euler equations for a linearly degenerate equation of state}, J. Hyperbolic Differ. Equ., {\bf 14} (2017), 535-563.


\end{thebibliography}
\end{document}